\documentclass{siamart0216}

\newsiamremark{remark}{Remark}

\usepackage{amsfonts,amsmath,amssymb}
\usepackage{mathrsfs,mathtools,stmaryrd,wasysym}
\usepackage{enumerate}
\usepackage{xspace} 
\usepackage{esint}
\usepackage{hyperref}
\DeclareMathAlphabet{\mathpzc}{OT1}{pzc}{m}{it}
\newcommand{\R}{\mathbb{R}}
\newcommand{\polS}{\mathbb{S}}
\newcommand{\polD}{\mathbb{D}}
\newcommand{\polA}{\mathbb{A}}
\newcommand{\calI}{\mathcal{I}}

\newcommand{\calJ}{\mathcal{J}}
\newcommand{\bu}{\mathbf{u}}
\newcommand{\bv}{\mathbf{v}}

\newcommand{\be}{\mathbf{e}}
\newcommand{\GRAD}{\nabla}
\newcommand{\bMcal}{\boldsymbol{\mathcal{M}}}
\newcommand{\calM}{\mathcal{M}}
\DeclareMathOperator{\DIV}{div}
\newcommand{\diff}{\, \mbox{\rm d}}
\newcommand{\ie}{i.e.,\@\xspace}

\DeclareMathOperator{\diam}{diam}
\newcommand{\bF}{{\mathbf{F}}}
\newcommand{\bC}{{\mathbf{C}}}
\newcommand{\bL}{{\mathbf{L}}}
\newcommand{\bW}{{\mathbf{W}}}
\newcommand{\bU}{{\mathbf{U}}}
\newcommand{\polG}{{\mathbb{G}}}
\newcommand{\polN}{{\mathbb{N}}}
\newcommand{\bG}{{\mathbf{G}}}
\newcommand{\bE}{{\mathbf{E}}}

\newcommand{\bef}{{\mathbf{f}}}
\newcommand{\bX}{{\mathbf{X}}}
\newcommand{\bvphi}{{\boldsymbol{\varphi}}}

\newcommand{\bpsi}{{\boldsymbol{\psi}}}

\newcommand{\bmu}{{\boldsymbol{\mu}}}

\newcommand{\bw}{{\mathbf{w}}}
\newcommand{\T}{\mathscr{T}_h}

\renewcommand{\P}{\mathcal{P}}
\newcommand{\vare}{\boldsymbol{\varepsilon}}
\newcommand{\cL}{\mathring{L}}
\newcommand{\bQ}{\mathbf{Q}}
\newcommand{\bP}{\mathbf{P}}
\newcommand{\bI}{\mathbf{I}}

\newcommand{\pe}{{\mathsf{p}}}

\DeclareMathOperator{\supp}{supp}

\DeclareMathOperator{\dist}{dist}
\DeclareMathOperator{\tr}{tr}

\newcommand{\frakG}{{\mathfrak{G}}}
\newcommand{\frakq}{{\mathfrak{q}}}
\newcommand{\calF}{{\boldsymbol{\mathcal{F}}}}
\newcommand{\Fortin}{{\mathcal{F}_h}}

\newcommand{\EO}[1]{#1}
\newcommand{\AJS}[1]{#1}

\newcommand{\TheTitle}{On the analysis and approximation of some models of fluids over weighted spaces on convex polyhedra}
\newcommand{\ShortTitle}{Fluids on convex polyhedra}
\newcommand{\TheAuthors}{E.~Ot\'arola, A. J.~Salgado}

\headers{\ShortTitle}{\TheAuthors}

\title{{\TheTitle}\thanks{EO has been partially supported by CONICYT through FONDECYT project 11180193. AJS has been partially supported by NSF grant DMS-1720213.}}

\author{
  Enrique Ot\'arola\thanks{Departamento de Matem\'atica, Universidad T\'ecnica Federico Santa Mar\'ia, Valpara\'iso, Chile.
    (\email{enrique.otarola@usm.cl}, \url{http://eotarola.mat.utfsm.cl/}).}
  \and
  Abner J.~Salgado\thanks{Department of Mathematics, University of Tennessee, Knoxville, TN 37996, USA.
    (\email{asalgad1@utk.edu}, \url{http://www.math.utk.edu/\string~abnersg})}
}

\ifpdf
\hypersetup{
  pdftitle={\TheTitle},
  pdfauthor={\TheAuthors}
}
\fi

\begin{document}

\maketitle

\begin{abstract}
We study the Stokes problem over convex polyhedral domains on weighted Sobolev spaces. The weight is assumed to belong to the Muckenhoupt class $A_q$ for $q \in (1,\infty)$. We show that \EO{the Stokes} problem is well-posed for all $q$. In addition, we show that the finite element Stokes projection is stable on weighted spaces. With the aid of these tools, we provide well-posedness and approximation results to some classes of non-Newtonian fluids.
\end{abstract}



\section{Introduction}
\label{sec:intro}

The purpose of this work is to study well-posedness and approximation results, on weighted spaces, for some models of non-Newtonian fluids on convex polyhedral domains. To be specific, we will study the following problem
\begin{equation}
\label{eq:nlStokes}
  \begin{dcases}
    -\DIV \polS(x,\vare(\bu)) + \GRAD \pe = -\DIV \bef, & \text{ in } \Omega, \\
    \DIV \bu = g, & \text{ in } \Omega, \\
    \bu = \boldsymbol0, & \text{ on } \partial \Omega.
  \end{dcases}
\end{equation}
\EO{Throughout our work,} we assume the domain $\Omega \subset \R^3$ to be a convex polyhedron. For a vector field $\bv$, we denote by $\vare(\bv) = \tfrac12 (\GRAD \bv + \GRAD \bv^\intercal)$ its symmetric gradient. Specific 
assumptions for the stress tensor $\polS$ and for the data $\bef$ and $g$ will be made explicit as we study particular models.

We must make an immediate comment regarding the space dimension. While the three dimensional case is the most significant from the physical point of view, our restriction to this case is of purely technical nature. This is because many of our results heavily rely on H\"older estimates for the derivatives of the Green matrix;
see section~\ref{sub:Green}. 
As far as we are aware, some of these \EO{estimates} are not available in the literature in either dimension two or higher than three. As soon as these become available, our results will readily extend to these dimensions as well.

The main source of difficulty and originality in this work can be summarized as follows. First, most of the well-posedness results for non-Newtonian fluids of the form \eqref{eq:nlStokes} are presented for domains that are at least $C^1$; see, for instance, \cite{MR2001659,MR3582412,MR2272870,MR2338412,MR2471134}. However, this assumption is not amenable to finite element discretization. For this reason, we focus on convex polyhedra. We are able to provide approximation results, over quasiuniform meshes, for each one of the models that we consider. Second, we allow the data to be singular. Even in the linear case, \ie $\polS(x,\vare) = 2\mu\vare$, for a constant $\mu>0$, the study of well-posedness results on convex polyhedra is far from being trivial. Here, by singular, we mean that $\bef \in \bL^q(\omega,\Omega)$ and $g \in \cL^q(\omega,\Omega)$, where $\omega \in A_q$ for $q \in (1,\infty)$; see section~\ref{sec:Prelim} for notation. This allows to have even measure valued forcings.
Finally, 
we may allow the constitutive relation $\polS(x,\vare)$ to be degenerate, as naturally appears when considering the Smagorinski model of turbulence described in section \ref{sec:Rappazz}. In such a setting, problem \eqref{eq:nlStokes}, once again, must be understood in suitably weighted Sobolev spaces.

The history of the analysis and approximation of classes of non-Newtonian fluids is too vast and deep to even attempt to provide a complete description here. It can be, for instance, traced back to the work of Lady\v{z}enskaja \cite{MR0155092,MR0241832,MR0226907}, and the famous model that now bears her name; see also \cite{MR1089323}. Other classes of non-Newtonian fluids that are similar to those we consider here have been studied in \cite{MR1602949,MR1348587}. Regarding approximation, to our knowledge, some of the first works that deal with finite element discretizations of non-Newtonian fluids are \cite{MR1034917,MR1069652,MR1211613}. The estimates of these works were later refined and improved in \cite{MR1213411,MR1301740}. This last reference introduced the concept of quasi--norm error bounds, and led to further developments using Orlicz spaces and shifted \emph{N}--functions that were proposed, for instance, in \cite{MR2914267}. Similar estimates, but via different arguments, were obtained in \cite{MR3042563}.

We organize our presentation as follows. In section~\ref{sec:Prelim} we present notation and gather some well-known facts that shall be useful for our purposes. In particular, we present regularity results for the classical Stokes problem in convex polyhedra and suitable H\"older estimates on the associated fundamental solution. We also recall some facts about the approximation properties of finite elements in standard and weighted Sobolev spaces. Section~\ref{sec:Casas} is our first original contribution. Via a duality argument we obtain an $L^2$-error estimate for the discretization of the Stokes problem when the forcing is a general Radon measure. This estimate improves our previous work \cite{DOS:19}. The fundamental solution estimates are used in section~\ref{sec:Apweights} to show that, for every $q \in (1,\infty)$ and every $\omega \in A_q$, the Stokes problem is well-posed on weighted spaces $\bW^{1,q}_0(\omega,\Omega) \times \cL^q(\omega,\Omega)$. The study of non-Newtonian models begins in section~\ref{sec:Bulicek2}, where we extend the well-posedness results of \cite{MR3582412} to the case of convex polyhedra and provide approximation results over quasiuniform meshes. Finally, in section~\ref{sec:Rappazz}, we study a variant of the well-known Smagorinski model of turbulence, which was originally developed in \cite{MR3488119}, and aims at reducing the well-known overdissipation effects that this model presents near walls \cite{MR3494304}. Existence of solution  and approximation results over quasiuniform meshes are obtained.

\section{Preliminaries}
\label{sec:Prelim}
\EO{We begin by fixing notation and the setting in which we will operate. Throughout our work $\Omega \subset \R^3$ is a convex polyhedron.}
For $w \in L^1(\Omega)$ and $D \subset \Omega$, we set
\[
  \fint_D w \diff x = \frac1{|D|} \int_D w \diff x,
  \qquad
  \EO{|D| = \int_{D} \diff x}.
\]
\EO{We shall use standard notation for Lebesgue and Sobolev spaces.} Spaces of vector valued functions and its elements will be indicated with boldface. Since we will mostly deal with incompressible fluids, we must indicate a way to make the pressure unique.
To do so, for $q\in[1,\infty)$ we denote by $\cL^q(\Omega)$ the space of functions in $L^q(\Omega)$ that have zero averages.

\EO{For a set $E \subset \R^3$ we denote its interior by $\mathring E$}. For a cube $Q$ with sides parallel to the coordinate axes we denote by $\ell(Q)$ the length of its sides. If $Q$ is a cube, and $a>0$,  we denote by $aQ$ the cube with same center but with sidelength $a \ell(Q)$.

\EO{We denote by $\polN_0 = \polN \cup \{0\}$ and by $\delta_{i,j}$ the Kronecker delta.} The relation $A \lesssim B$ indicates that there is a nonessential constant $c$ such that $A \leq c B$. By $A \approx B$ we mean $A \lesssim B \lesssim A$. Whenever $q \in (1,\infty)$, we indicate by $q'$ its H\"older conjugate.

\subsection{Weights} 
One of the tools that will allow us to deal with singular sources, and nonstandard rheologies, is the use of weighted spaces and weighted norm inequalities. A weight is an almost everywhere positive function in $L^1_{\mathrm{loc}}(\R^3)$. Let $q \in [1,\infty)$, we say that a weight $\omega$ is in the Muckenhoupt class $A_q$ if \cite{MR1800316,MR2491902,MR1774162}
\begin{equation}
\begin{aligned}
\label{A_pclass}
\left[ \omega \right]_{A_q} & := \sup_{B} \left( \fint_{B} \omega \diff x\right) \left( \fint_{B} \omega^{1/(1-q)} \diff x \right)^{q-1}  < \infty, \quad q \in (1,\infty),
\\
\left[ \omega \right]_{A_1} & := \sup_{B} \left( \fint_{B} \omega \diff x\right)  \sup_{x \in B} \frac{1}{\omega(x)}< \infty, \quad q =1,
\end{aligned}
\end{equation}
where the supremum is taken over all balls $B$ in $\R^3$. We call $[\omega]_{A_q}$, for $q \in [1,\infty)$, the Muckenhoupt characteristic of the weight $\omega$. We observe that, for $q \in (1,\infty)$, there is a certain conjugacy in the $A_q$ classes: $\omega \in A_q$ if and only if $\omega' = \omega^{1/(1-q)} = \EO{\omega^{1-q'} }\in A_{q'}$ \EO{\cite[Proposition 7.2 (2)]{MR1800316}. Finally, we note that $A_p \subset A_q$ for $1 \leq p < q$ \cite[Proposition 7.2 (i)]{MR1800316}. In particular, we have that} $A_1 \subset A_q$ for all $q >1$.

\AJS{Let $q \in (1,\infty)$ and $\omega \in A_q$. We define the space $\cL^q(\omega,\Omega) = L^q(\omega,\Omega) \cap \cL^1(\Omega)$, where $L^q(\omega,\Omega)$ denotes the Lebesgue space \EO{of $q$-integrable functions} with respect to the measure $\omega \diff x$. Weighted Sobolev spaces are defined accordingly.} On weighted spaces the following inf-sup condition holds \cite[Lemma 6.1]{DOS:19}
\begin{equation}
\label{eq:infsuppres}
  \| p \|_{L^q(\omega,\Omega)} \lesssim \sup_{\boldsymbol0 \neq \bv \in \bW^{1,q'}_0(\omega',\Omega)} \frac{ \int_\Omega p \DIV \bv \diff x}{ \| \GRAD \bv \|_{\bL^{q'}(\omega',\Omega)}} \quad \forall p \in \cL^q(\omega,\Omega).
\end{equation}
This estimate will become useful in the sequel. \EO{On the other hand,} \AJS{the following weighted version of Korn's inequality holds \cite[Theorem 5.15]{MR2643399}
\begin{equation}
\label{eq:WeightedKorn}
  \| \GRAD \bv \|_{\bL^q(\omega,\Omega)} \lesssim \| \vare(\bv) \|_{\bL^q(\omega,\Omega)}, \quad \forall \bv \in \bW^{1,q}_0(\omega,\Omega).
\end{equation}

Let} $z \in \Omega$ be an interior point of $\Omega$ and $\alpha \in \R$. Define
\begin{equation}
\label{distance_A2}
{\textup{\textsf{d}}}_{z} ^\alpha(x) = |x-z|^{\alpha}.
\end{equation}
The weight ${\textup{\textsf{d}}}_{z} ^\alpha \in A_2$ provided that $\alpha \in (-3,3)$. Notice that there is a neighborhood of $\partial \Omega$ where ${\textup{\textsf{d}}}_{z} ^\alpha$ has no degeneracies or singularities. This observation motivates us to define a restricted class of Muckenhoupt weights \cite[Definition 2.5]{MR1601373}.

\begin{definition}[class $A_q(\Omega)$]
\label{def:ApOmega}
Let $\Omega \subset \R^3$ be a Lipschitz domain and $q \in [1,\infty)$. We say that $\omega \in A_q$ belongs to $A_q(\Omega)$ if there is an open set $\mathcal{G} \subset \Omega$ and $\varepsilon, \omega_l >0$ such that:
\[
\{ x \in \Omega: \mathrm{dist}(x,\partial\Omega)< \varepsilon\} \subset \mathcal{G},
\qquad
\omega |_{ \bar{\mathcal{G}} } \in C(\bar{\mathcal{G}}),
\qquad
\omega_l \leq \omega(x) \quad \forall x \in \bar{\mathcal{G}}.
\]
\end{definition}

In \cite{OS:17infsup} it was shown that, provided the weight belongs to this class, the Stokes problem is well-posed on weighted spaces and Lipschitz domains. One of the highlights of this work is that we, in a sense, remove this restriction on the weight \EO{at the expense of assuming the convexity of the domain.}

\subsection{Maximal operators}
For $w \in L^1_{\mathrm{loc}}(\AJS{\R^3})$, the Hardy--Littlewood maximal operator is defined by \cite[Chapter 7, \EO{Section 1}]{MR1800316}
\begin{equation}
\label{eq:Maximal}
 \mathcal{M} w(x) = \sup_{Q \ni x} \fint_{Q} |w(y)| \diff y,
\end{equation}
where the supremum is taken over all cubes $Q$ containing $x$. One of the main properties of the Muckenhoupt classes $A_q$ \EO{previously} introduced is that, for $q \in (1,\infty)$, the maximal operator $\calM$ is continuous on $L^q(\omega,\Omega)$ \cite[Theorem 7.3]{MR1800316}.

We will also make use of the sharp maximal operator, which is defined, for $w \in L^1_{\mathrm{loc}}(\Omega)$, by \cite[Chapter 6, Section 2]{MR1800316}
\begin{equation}
\label{eq:Maximal_sharp}
 \calM^\sharp_\Omega w(x) = \sup_{Q \ni x} \fint_{Q} \left|w(y) - \fint_Q w(z) \diff z \right| \diff y.
\end{equation}
The supremum is taken over all cubes $Q \subset \Omega$ containing $x$. It is important to notice that, when bounding the sharp maximal operator it suffices to bound the difference between $w$ and any constant $c$ \EO{\cite[Proposition 6.5]{MR1800316}}. In fact,
\begin{equation}
  \begin{aligned}
    \int_{Q}  \left| w(y) - \fint_Q w(z) \diff z \right| \diff y
    & \leq 
    \int_{Q}  \left| w(y) - c \right| \diff y
    +
    \int_{Q}  \left| c - \fint_Q w(z) \diff z \right| \diff y
    \\
    & \leq 2 \int_{Q}  \left| w(y) - c \right| \diff y.
  \end{aligned}
\label{eq:average_does_not_matter}
\end{equation}

\subsection{The Stokes problem}
\label{sub:Stokesprelim}

\EO{In this section}, we collect some facts on the Stokes problem that are well-known and will be used repeatedly. \EO{Let us} set the constitutive relation $\polS(x,\vare) = 2\mu\vare$, where $\mu>0$. Problem \eqref{eq:nlStokes} thus becomes
\begin{equation}
\label{eq:Stokes}
  \begin{dcases}
    -\DIV(2\mu\vare(\bu)) + \GRAD \pe = \bF, & \text{ in } \Omega, \\
    \DIV \bu = g, & \text{ in } \Omega, \\
    \bu = \boldsymbol0, & \text{ on } \partial\Omega.
  \end{dcases}
\end{equation}
Problem \eqref{eq:Stokes} has a unique solution provided $\bF \in \bW^{-1,2}(\Omega)$ and $g \in \cL^2(\Omega)$, with a corresponding estimate. We also have the following regularity result.

\begin{proposition}[regularity]
\label{prop:regular}
Let $\Omega \subset \R^3$ be a convex polyhedron. If $\bF \in \bL^2(\Omega)$ and $g = 0$, then the solution to \eqref{eq:Stokes} is such that
\[
  \| \bu \|_{\bW^{2,2}(\Omega)} + \| \pe \|_{\AJS{W^{1,2}(\Omega)}} \lesssim \| \bF \|_{\bL^2(\Omega)}.
\]
\end{proposition}
\begin{proof}
See \EO{\cite[Theorem 2]{MR0404849}}, \cite{MR977489}, and \cite{MR1301452}; 
see also \cite[Corollary 1.8]{MR2987056}.
\end{proof}

Problem \eqref{eq:Stokes} is also well-posed in $L^q$ spaces.

\begin{theorem}[well-posedness in $L^q$]
\label{thm:WPLq}
Let $q\in (1,\infty)$ and $\Omega \subset \R^3$ be a convex polyhedron. If $\bF \in \bW^{-1,q}(\Omega)$ and $g \in \cL^q(\Omega)$, then problem \eqref{eq:Stokes} has a unique solution $(\bu, \pe) \in \bW^{1,q}_0(\Omega) \times \cL^q(\Omega)$ that satisfies the estimate
\[
  \| \vare(\bu) \|_{\bL^q(\Omega)} + \| \pe \|_{L^q(\Omega)} \lesssim \| \bF \|_{\bW^{-1,q}(\Omega)} + \| g \|_{L^q(\Omega)}.
\]
\AJS{In particular, if $\bF = -\DIV \bef$ with $\bef \in \bL^q(\Omega)$, we have
\begin{equation}
  \| \vare(\bu) \|_{\bL^q(\Omega)} + \| \pe \|_{L^q(\Omega)} \lesssim \| \bef \|_{\bL^q(\Omega)} + \| g \|_{L^q(\Omega)}.
  \label{eq:unweighted_F_as_DIVf}
\end{equation}
In both estimates, the hidden constants are independent of $\bF$, $g$, $\bu$ and $\pe$.}
\end{theorem}
\begin{proof}
Evidently, we only need to comment on the case $q \neq 2$. For a proof of the result when $q > 2$, we refer the reader to the first item in Section 5.5 of \cite{MR2321139}. Using the equivalent characterization of well-posedness via inf-sup conditions, one can deduce well-posedness for $q \in (1,2)$. \AJS{Finally, the inequality
\[
  \| \DIV\bef \|_{\bW^{-1,q}(\Omega)} \leq \| \bef \|_{\bL^{q}(\Omega)},
\]
immediately yields the second estimate.}
\end{proof}

\begin{remark}[equivalence]
\label{rem:LapISDivVare}
In the literature, the Stokes problem is usually presented with the term $\DIV(2\mu\vare(\bu))$ replaced by $\mu\Delta\bu$. Using the elementary identity
\[
  \DIV(2\vare(\bv)) = \Delta \bv + \GRAD \DIV \bv,
\]
it is not difficult to see that this only amounts to a redefinition of the pressure. This redefinition, however, does not affect the conclusions of Proposition~\ref{prop:regular} or Theorem~\ref{thm:WPLq}.
\end{remark}

\subsection{The Green matrix}
\label{sub:Green}

We introduce the Green matrix $\polG : \bar\Omega \times \Omega \to \R^{4\times 4}$ for problem \eqref{eq:Stokes} as follows \EO{\cite[Section 11.5]{MR2641539}, \cite{MR2808700}}. Let $\phi \in C_0^\infty(\Omega)$ be such that 
\[
  \int_\Omega \phi \diff x = 1.
\]
We represent the entries of $\polG$ as
\[
  \polG = \begin{pmatrix} \bG_1 & \bG_2 & \bG_3 & \bG_4 \\ \lambda_1 & \lambda_2 & \lambda_3 & \lambda_4 \end{pmatrix}, \quad \bG_j = (\bG_{1,j}, \bG_{2,j}, \bG_{3,j})^\intercal, \quad j \in \{1,2,3, \AJS{4}\},
\]
where the pairs $(\bG_j,\lambda_j)_{j=1}^4$ are distributional solutions of
\[
  \begin{dcases}
    -2\DIV_x( \vare_x(\bG_j(x,\xi))) + \GRAD_x \lambda_j(x,\xi) = \delta(x-\xi) \be_j, & x,\xi \in \Omega, \\
    \DIV_x \bG_j(x,\xi) = 0, & x,\xi \in \Omega, \\
    \bG_j(x,\xi) = \boldsymbol0, & x \in \partial\Omega, \ \xi \in \Omega
  \end{dcases}
\]
for $j=1,2,3$ and
\[
  \begin{dcases}
    -2\DIV_x(\vare_x( \bG_4(x,\xi))) + \GRAD_x \lambda_4(x,\xi) = 0, & x,\xi \in \Omega, \\
    -\DIV_x \bG_4(x,\xi) = \delta(x-\xi) - \phi(x), & x,\xi \in \Omega, \\
    \bG_4(x,\xi) = \boldsymbol0, & x \in \partial\Omega, \ \xi \in \Omega.
  \end{dcases}
\]
Here, $\{\be_j\}_{j=1}^3$ denotes the canonical basis of $\R^3$ and $\delta$ the Dirac distribution. For uniqueness, we also require that
\[
  \int_\Omega \lambda_j(x,\xi) \phi(x) \diff x =0, \quad \xi \in \Omega, \quad j = 1, \ldots, 4.
\]
The existence and uniqueness of the Green matrix $\polG$ follows from \EO{\cite[Theorem 11.4.1]{MR2641539}}. Note that $\polG_{i,j}(x,\xi) = \polG_{j,i}(\xi,x)$ for $x, \xi \in \Omega$ and $i,j \in \{1, \ldots, 4\}$ \EO{\cite[Theorem 11.4.1]{MR2641539}}. The importance of this matrix lies in the fact that it provides a representation formula for the solution of \eqref{eq:Stokes}. In particular, we have that 
\begin{equation}
\label{eq:representU}
  \bu_j(\xi) = \frac1\mu \langle \bF,\bG_j(\cdot,\xi) \rangle - \int_\Omega \lambda_j(x,\xi) g(x) \diff x, \quad j \in \{1,2,3\},
\end{equation}
where $\langle \cdot,  \cdot\rangle$ denotes a suitable duality pairing. This representation shall become useful in the sequel.

\EO{The following estimates for  the Green matrix will be essential in what follows.}

\begin{theorem}[H\"older estimates]
\label{thm:HolderGreen}
Let $\Omega \subset \R^3$ be a convex polyhedron.
\EO{There exists} $\sigma \in (0,1)$, that depends only on the domain, such that for any multiindices $\alpha, \beta \in \polN_0^3$, the Green matrix $\polG$ satisfies
\[
  \left| \partial_x^\alpha \partial_\xi^\beta \polG_{i,j}(x,\xi) - \partial_y^\alpha \partial_\xi^\beta \polG_{i,j}(y,\xi) \right| \lesssim |x-y|^\sigma 
  \left( |x-\xi|^{-a} + |y-\xi|^{-a} \right),
\]
for $x,y,\xi \in \Omega$ with $x \neq y$ and
\[
  \left| \partial_x^\alpha \partial_\xi^\beta \polG_{i,j}(x,\xi) - \partial_x^\alpha \partial_\eta^\beta \polG_{i,j}(x,\eta) \right| \lesssim |\xi-\eta|^\sigma 
  \left( |x-\xi|^{-a} + |x - \eta|^{-a} \right),
\]
for $x,\xi, \eta\in \Omega$ with $\xi \neq \eta$, whenever $|\alpha|\leq 1-\delta_{i,4}$ and $|\beta|\leq 1-\delta_{j,4}$. Here, $i,j \in \{1,2,3,4\}$ and
\[
  a = 1+\sigma +\delta_{i,4}+\delta_{j,4}+|\alpha|+|\beta|.
\]
\end{theorem}
\begin{proof}
\EO{References \cite{MR2641539,MR2808700} provide} the claimed estimates for the Green's matrix for the case where the operator acting on the $\bG$--components of the matrix is the Laplacian. It suffices to proceed with the change of variables described in Remark~\ref{rem:LapISDivVare} and observe that the $\lambda$--components of the Green's matrix have the same differentiability properties as derivatives of the $\bG$--components.
\end{proof}

\subsection{Finite elements}
\label{sub:FEM}

Many of the results we wish to discuss involve error estimates for finite element schemes. \AJS{Since our domain $\Omega$ is a convex polyhedron, it can be triangulated exactly. We thus assume that, for every $h>0$, we have at hand a quasiuniform, in the sense of \cite{MR851383,CiarletBook,Guermond-Ern}, triangulation $\T$ of the domain $\Omega$.} \EO{We construct, over these triangulations,} finite dimensional spaces $\bX_h \times M_h \subset \bW^{1,\infty}(\Omega) \times ( L^\infty(\Omega) \cap \cL^2(\Omega))$ that satisfy, for every $q \in (1,\infty)$ and $\omega \in A_q$, the compatibility condition
\begin{equation}
\label{eq:infsuph}
  \| p_h \|_{L^{q'}(\omega',\Omega)} \lesssim \sup_{\bv_h \in \bX_h} \frac{ \int_\Omega \DIV \bv_h p_h \diff x}{ \| \GRAD \bv_h \|_{\bL^q(\omega,\Omega)}} \quad \forall p_h \in M_h.
\end{equation}
Moreover, we require that these spaces have \AJS{the usual} approximation properties 
over quasiuniform meshes of size $h$. In particular, we require the existence of a \emph{stable} operator $\calI_h: \bL^1(\Omega) \to \bX_h$ that preserves the space $\bX_h$ and satisfies the error estimates
\begin{equation}
\label{eq:interpolant}
  \| \bv - \calI_h \bv \|_{\bL^\infty(\Omega)} \lesssim h^{1/2} \| \bv \|_{\bW^{2,2}(\Omega)} \quad \forall \bv \in \bW^{1,2}_0(\Omega) \cap \bW^{2,2}(\Omega),
\end{equation}
and, for $q \in (1,\infty)$ and $\omega \in A_q$,
\begin{equation}
\label{eq:interpolantWeights}
  \| \GRAD(\bv - \calI_h \bv) \|_{\bL^q(\omega,\Omega)} \lesssim \inf_{\bv_h \in \bX_h} \| \GRAD(\bv - \bv_h) \|_{\bL^q(\omega,\Omega)} \quad \EO{\forall \bv \in \bW^{1,q}_0(\omega,\Omega).}
\end{equation}

We finally comment that, since the continuous inf-sup condition \eqref{eq:infsuppres} holds, \eqref{eq:infsuph} is equivalent to the existence of a so--called \emph{Fortin operator} \cite[Lemma 4.19]{Guermond-Ern}, that is an operator $\Fortin : \bW^{1,q}_0(\omega,\Omega) \to \bX_h$ that preserves the divergence, \ie
\[
  \int_\Omega r_h \DIV( \bv - \Fortin \bv) \diff x = 0 \quad \forall \bv \in \bW^{1,q}_0(\omega,\Omega), \quad \forall r_h \in M_h,
\]
\EO{preserves the space $\bX_h$ and is stable \cite{MR3071173}.} This immediately implies that $\Fortin$ possesses quasi-optimal approximation properties, \ie for all $\bv \in \bW^{1,q}_0(\omega,\Omega)$,
\[
  \| \GRAD(\bv - \Fortin \bv) \|_{\bL^q(\omega,\Omega)} \lesssim \inf_{\bv_h \in \bX_h} \| \GRAD(\bv - \bv_h) \|_{\bL^q(\omega,\Omega)}.
\]

\AJS{Within the unweighted setting, examples of such pairs are well-known in the literature \cite{MR851383,Guermond-Ern}. For extensions to the weighted case, we refer the reader to \cite[Section 6]{DOS:19}.
An operator satisfying \eqref{eq:interpolant} and \eqref{eq:interpolantWeights} has been constructed in \cite{NOS3}. \EO{We mention that, usually, $\bX_h$ and $M_h$}
consists of piecewise polynomials subject to $\T$. }

Given $(\bu,p) \in W^{1,1}_0(\Omega) \times \cL^1(\Omega)$, with $\DIV \bu =0$, we define its \emph{Stokes projection} as the pair $(\bu_h,p_h) \in \bX_h \times M_h$ that satisfies
\begin{equation}
\label{eq:defofStokesprojection}
  \begin{dcases}
    \AJS{2}\mu \int_\Omega \vare (\bu- \bu_h): \vare(\bv_h) \diff x - \int_\Omega (p -p_h) \DIV \bv_h \diff x = 0 & \forall \bv_h \in \bX_h, \\
    \int_\Omega \DIV(\bu - \bu_h) r_h \diff x = 0 & \forall r_h \in M_h.
  \end{dcases}
\end{equation}

We recall that, under the given assumptions on the finite element spaces, the Stokes projection is stable on weighted spaces.

\begin{theorem}[weighted stability estimate]
\label{thm:stabclassic}
Let $\Omega \subset \R^3$ be a convex polyhedron. If $\omega \in A_1$ then, the finite element Stokes projection, defined in \eqref{eq:defofStokesprojection}, is stable in $\bW^{1,2}_0(\omega,\Omega) \times \cL^2(\omega,\Omega)$, in the sense that
\[
  \| \vare(\bu_h) \|_{\bL^2(\omega,\Omega)} + \| p_h \|_{L^2(\omega,\Omega)} \lesssim \| \vare(\bu) \|_{\bL^2(\omega,\Omega)} + \| p \|_{L^2(\omega,\Omega)},
\]
where the hidden constant is independent of $h$, $\bu$ and $p$.
\end{theorem}
\begin{proof}
The proof follows after small modifications to \cite[Theorem 4.1]{DOS:19}; see Appendix~\ref{sec:GNS} for details.
\end{proof}

\section{An error estimate in $L^2$}
\label{sec:Casas}

In this section, we discuss error estimates for discretizations of \eqref{eq:Stokes} in the case $g = 0$ and $\bF = \bmu \in \bMcal_b(\Omega)$, the space of vector valued Radon measures. In doing so, we shall extend the results of \cite{MR812624} to the Stokes problem, and slightly improve the error estimate of \cite[Corollary 5.4]{DOS:19}.

Let $q \in (1,3/2)$. Since $q'>3$, we have that $W^{1,q'}_0(\Omega) \hookrightarrow C(\bar\Omega)$. Therefore,
\[
  \calM_b(\Omega) = ( C_0(\bar\Omega) )' \hookrightarrow ( W^{1,q'}_0(\Omega) )' = W^{-1,q}(\Omega).
\]
Invoking Theorem~\ref{thm:WPLq}, we have that problem \eqref{eq:Stokes} is well-posed for such data. If, in addition, we assume that \EO{$q \geq 6/5$}, then we also have that $W^{1,q}(\Omega) \hookrightarrow L^2(\Omega)$. As a consequence, it makes sense to provide an error estimate in $L^2(\Omega)$. Our main result in this direction is the following.

\begin{theorem}[error estimate]
Let $\Omega \subset \R^3$ be a convex polyhedron, $q \in \EO{[}6/5,3/2)$, $\bF = \bmu \in \bMcal_{b}(\Omega)$, and $g=0$. Let $(\bu,\pe) \in \bW^{1,q}_0(\Omega) \times \cL^q(\Omega)$ be the solution to \eqref{eq:Stokes} and $(\bu_h,\pe_h) \in \bX_h \times M_h$ its Stokes projection as defined in \eqref{eq:defofStokesprojection}. Then, we have
\[
  \| \bu - \bu_h \|_{\bL^2(\Omega)} \lesssim h^{1/2} \| \bmu \|_{\bMcal_b(\Omega)},
\]
with a hidden constant independent of $h$, $(\bu,\pe)$, $(\bu_h,\pe_h)$, and $\bmu$.
\end{theorem}
\begin{proof}
For $\bw \in \bL^2(\Omega)$, we let $(\bvphi,\psi) \in \bW^{1,2}_0(\Omega) \times \cL^2(\Omega)$ be the solution of
\begin{equation}
\label{eq:Stokes_for_duality}
\AJS{2}\mu \int_{\Omega} \vare(\bvphi) : \vare(\bv) \diff x  - \int_{\Omega} \psi \DIV \bv \diff x = \int_\Omega \bw \cdot \bv \diff x,
\qquad
\int_{\Omega} r \DIV \bvphi \diff x= 0 
\end{equation}
for all $\bv \in  \bW^{1,2}_0(\Omega)$ and $r \in \cL^2(\Omega)$. Since $\Omega $ is convex \EO{and $\bw \in \bL^2(\Omega)$,} owing to Proposition~\ref{prop:regular} we have that $( \bvphi, \psi) \in \bW^{2,2}(\Omega) \times W^{1,2}(\Omega)$. This, in particular, implies that if $(\bvphi_h,\psi_h) \in \bX_h \times M_h$ denotes its Stokes projection, we have
\begin{equation}
\label{eq:l_inf_estimate_dual}
\begin{aligned}
  \| \bvphi - \bvphi_h \|_{\bL^\infty(\Omega)} &\leq  \| \bvphi - \calI_h \bvphi  \|_{\bL^\infty(\Omega)}  +   \| \calI_h \bvphi  - \bvphi_h \|_{\bL^\infty(\Omega)}
  \\
  &\lesssim h^{1/2} \| \bvphi \|_{\bW^{2,2}(\Omega)} + h^{-3/2} \| \calI_h \bvphi  - \bvphi_h \|_{\bL^2(\Omega)} \\ &\lesssim h^{1/2} \| \bvphi \|_{\bW^{2,2}(\Omega)} \lesssim h^{1/2} \| \bw \|_{\bL^2(\Omega)},
\end{aligned}
\end{equation}
where we used the interpolation error estimate \eqref{eq:interpolant}, a basic inverse inequality, the invariance property of  $\calI_h$, the stability of $\calI_h$ in $\bL^2(\Omega)$, an interpolation estimate for $\calI_h$ in $\bL^2(\Omega)$, and the regularity results of Proposition \ref{prop:regular}.

Consider now problem \eqref{eq:Stokes_for_duality} with $\bw = \bu - \bu_h$ and set $\bv =  \bu-\bu_h \in \bW^{1,2}_0(\Omega)$. This immediately yields
\begin{equation}
  \| \bu-\bu_h \|_{\bL^2(\Omega)}^2 = \AJS{2}\mu \int_{\Omega} \vare(\bvphi) : \vare(\bu-\bu_h) \diff x- \int_{\Omega} \psi \DIV( \bu-\bu_h) \diff x.
  \label{eq:basic_estimate}
\end{equation}

Observe that, since $(\bvphi_h,\psi_h) \in \bX_h \times M_h$ corresponds to the Stokes projection of $(\bvphi,\psi)$ and the pair $(\bu_h, \pe_h) \in \bX_h \times M_h$, we have
\[
 \AJS{2}\mu \int_{\Omega} \vare(\bvphi - \bvphi_h) : \vare(\bu_h) \diff x - \int_{\Omega} (\psi - \psi_h) \DIV \bu_h \diff x= 0.
\]
Similarly,
\[
 \AJS{2}\mu \int_{\Omega} \vare(\bu- \bu_h) : \vare(\bvphi_h) \diff x- \int_{\Omega} (\pe - \pe_h) \DIV \bvphi_h \diff x= 0.
\]

On the other hand, since $\bvphi \in \bW^{2,2}(\Omega)$ and $\bvphi_h \in \bW^{1,\infty}(\Omega)$, then we have that $\bvphi - \bvphi_h \in \bW^{1,s}(\Omega)$ for every $s \leq 6$. We can thus consider $\bvphi - \bvphi_h$ as a test function in the weak version of \eqref{eq:Stokes} with $\bF = \bmu$ to arrive at
\[
 \AJS{2}\mu \int_{\Omega} \vare(\bu) :  \vare(\bvphi - \bvphi_h) \diff x - \int_{\Omega} \pe \DIV (\bvphi - \bvphi_h)  \diff x= \int_{\Omega}  (\bvphi - \bvphi_h) \cdot  \diff  \bmu(x).
\]

We can then rewrite the right hand side of \eqref{eq:basic_estimate} and invoke the previous three relations to conclude, on the basis of \eqref{eq:l_inf_estimate_dual}, that
\begin{align*}
  \| \bu-\bu_h \|_{\bL^2(\Omega)}^2 &= \int_\Omega (\bvphi - \bvphi_h) \cdot \diff \bmu(x) \leq \| \bvphi - \bvphi_h \|_{\bL^\infty(\Omega)} \| \bmu \|_{\bMcal_b(\Omega)} \\
  &\lesssim h^{1/2} \| \bu - \bu_h \|_{\bL^2(\Omega)} \| \bmu \|_{\bMcal_{\AJS{b}}(\Omega)},
\end{align*}
which is the announced result. This concludes the proof.
\end{proof}

\section{Well-posedness of the Stokes problem on weighted spaces}
\label{sec:Apweights}

Let us now consider the Stokes problem \eqref{eq:Stokes} with $\bF = -\DIV \bef$, $\bef \in \bL^q(\omega,\Omega)$, and $\omega \in A_q$. We begin by recalling that, for general Lipschitz domains, there is $\epsilon = \epsilon(\Omega) >0$ such that, if $|q-2|<\epsilon$, and $\omega$ belongs to the restricted class $A_q(\Omega)$, then this problem is well-posed; see \cite[Theorem 17]{OS:17infsup} \EO{and \cite[Proposition 2.4 and Remark 2.5]{DOS:19}}. On the other hand, if $\Omega$ is $C^1$, then \cite[Lemma 3.2]{MR3582412} shows well-posedness for all $q \in (1,\infty)$ and all $\omega \in A_q$. Here we will show a result that, in a sense, is intermediate between these two. We remove the restriction on the integrability index $q$ and the boundary behavior of the weight, thus showing well-posedness for all $q \in (1,\infty)$ and all $\omega \in A_q$, but at the expense of requiring that $\Omega$ is a convex polyhedron.

The main tool that we shall use is the representation of the velocity given in \eqref{eq:representU} and the H\"older estimates of the Green matrix described in Theorem~\ref{thm:HolderGreen}. We will  follow the ideas of \cite{DDO:17}, and extend the results therein to the Stokes problem.

We begin by noting that, by density, it suffices to assume that $\bef \in \bC_0^\infty(\Omega)$, so that from \eqref{eq:representU} we can write
\begin{equation}
\label{eq:newptwise}
  \bu_j(x) = \frac1\mu \int_\Omega \GRAD_y \bG_j(y,x) \EO{:} \bef(y) \diff y - \int_\Omega \lambda_j(y,x) g(y) \diff y.
\end{equation}
We begin with a simplified version of the Bogovski{\u\i} decomposition of a function with integral zero \cite{MR553920}; see also \cite{MR3198867}. Since, in our setting, the proof of this result is so simple, we include it for completeness.

\begin{lemma}[decomposition]
\label{lem:Bogovskiidecompose}
Let $q\in (1,\infty)$, $\omega \in A_q$, $g \in \cL^q(\omega,\Omega)$, and $Q \subset \Omega$ be \AJS{a cube with sides parallel to the coordinate axes and} such that $\tfrac32 Q \subset \Omega$. Then, there are $g_i \in \cL^q(\omega,\Omega)$, with $i \in \{ 1,2 \}$, such that
\[
  g = g_1 + g_2, 
  \quad 
  \supp g_1 \subset \frac32 Q, 
  \quad 
  \supp g_2 \subset Q^c,
\]
\AJS{and
\[
  \| g_1 \|_{L^q(\omega,\frac32Q)} \lesssim \| g \|_{L^q(\omega,\frac32Q)},
  \qquad
  \| g_2 \|_{L^q(\omega,\Omega)} \lesssim \| g \|_{L^q(\omega,\Omega)},
\]
where the hidden constants are} independent of $g$ and $Q$.
\end{lemma}
\begin{proof}
To simplify notation let us set $D= \tfrac32 Q$ and $A = D \setminus Q$. Notice also that
\[
  |D| = \left(\frac32\right)^3 |Q|, \qquad |A| = |D| - |Q| = \left[ \left(\frac32\right)^3 -1 \right]|Q|,
\]
so that $|D|\approx|A|\approx|Q|$.

Let now $\phi \in C_0^\infty(\R^3)$ be such that $0 \leq \phi \leq 1$, $\phi \equiv 1$ on $Q$ and $\phi \equiv 0$ on $D^c$. Set 
\[
  g = \phi g + (1-\phi)g =: \tilde g_1 + \tilde g_2.
\]
Note that the functions $\tilde g_i$, for $i=1,2$, have the requisite support property. In addition,
\[
  \| \tilde g_1 \|_{L^q(\omega,\AJS{D})} \leq \| \phi \|_{L^\infty(\AJS{\R^3})} \| g \|_{L^q(\omega,\AJS{D})}, \quad
  \| \tilde g_2 \|_{L^q(\omega,\Omega)} \leq \| 1-\phi \|_{L^\infty(\AJS{\R^3})} \| g \|_{L^q(\omega,\Omega)}.
\]
The functions $\tilde g_1, \tilde g_2$, however, do not integrate to zero. Thus, we correct them as follows. Define
\[
  g_1 = \tilde g_1 - \frac{\chi_A}{|A|} \int_{D} \tilde g_1 \diff x.
\]
Then, we have that $\supp g_1 \subseteq \supp \tilde g_1 \cup A \subset D$ and, moreover,
\[
  \int_\Omega g_1 \diff x = \int_\Omega \tilde g_1 \diff x - \int_{D} \tilde g_1 \diff x = 0.
\]
Using that
\begin{equation}
\label{eq:avgtoweightLq}
   \left| \int_{D} \tilde g_1 \diff x \right| \leq \left( \int_D \omega^{-\frac{q'}{q}} \diff x \right)^{{\frac{1}{q'}}} \| \tilde g_1 \|_{L^q(\omega,\AJS{D})} \leq 
   \left( \int_D \omega^{-\frac{q'}{q}} \diff x \right)^{{\frac{1}{q'}}} \| g\|_{L^q(\omega,\AJS{D})},
\end{equation}
we are able to obtain the estimates
\begin{align*}
  \| g_1 \|_{L^q(\omega,\AJS{D})} &\leq \| \tilde g_1 \|_{L^q(\omega,\AJS{D})} 
  + 
  \frac{ 1}{|A|} \left(\int_A \omega\diff x\right)^{1/q} \left| \int_{D} \tilde g_1 \diff x \right| 
  \\
  &\leq \left[ 1 + \frac1{|A|} \left[ \left( \int_A \omega \diff x \right) \left( \int_D\omega^{-\frac{q'}{q}} \diff x \right)^{q-1} \right]^{1/q} \right] \| g \|_{L^q(\omega,\AJS{D})}.
\end{align*}
Now, since
\begin{align*}
  \frac1{|A|} \left[  \left(\int_A \omega \diff x\right) \left( \int_D\omega^{-\frac{q'}{q}} \diff x \right)^{q-1} \right]^{\frac{1}{q}}
  &\leq
  \frac{|D|}{|A|} \left[ \left(\fint_D\omega \diff x\right) \left( \fint_D\omega^{-\frac{q'}{q}} \diff x \right)^{q-1} \right]^{\frac{1}{q}} \\
  &\lesssim [\omega]_{A_q}^{\frac{1}{q}},
\end{align*}
we obtain \AJS{the local estimate}
$
  \| g_1 \|_{L^q(\omega,\AJS{D})} \lesssim \| g \|_{L^q(\omega,\AJS{D})},
$
with a constant that only depends on $[\omega]_{A_q}$.

Note now that the function
\[
  g_2 = \tilde g_2 + \frac{\chi_A}{|A|} \int_{D} \tilde g_1 \diff x
\]
satisfies $\supp g_2 \subseteq \supp \tilde g_2 \cup A \subset Q^c$, $g_1 + g_2 = \tilde g_1 + \tilde g_2 = g$, and, since $g \in \cL^q(\omega,\Omega)$,
\[
  \int_\Omega g_2 \diff x = \int_\Omega (g - g_1) \diff x = 0.
\]
Finally, using \eqref{eq:avgtoweightLq}, we have
\[
  \| g_2 \|_{L^q(\omega,\Omega)} \leq \| \tilde g_2 \|_{L^q(\omega,\Omega)} + \left| \int_{D} \tilde g_1 \diff x \right| \frac{ 1 }{|A|} \left(\int_A \omega\diff x\right)^{1/q} \lesssim \| g \|_{L^q(\omega,\Omega)},
\]
where the hidden constant only depends on $[\omega]_{A_q}$.

This concludes the proof.
\end{proof}

With this decomposition at hand, we can obtain an a priori estimate on the oscillation of the gradient of $\bu$, much as in \cite[Lemma 2.4]{DDO:17} and \cite[Lemma 7.9]{MR1800316}. 

\begin{lemma}[oscillation estimate]
\label{lem:lem24Duran}
Let $\Omega \subset \R^3$ be a convex polyhedron, $q \in (1,\infty)$, $\omega \in A_q$, $\bef \in \bL^q(\omega,\Omega)$, and $g \in \cL^q(\omega,\Omega)$. Let $\bu$ be the velocity component of the solution of \eqref{eq:Stokes} with $\bF = - \DIV \bef$. Then, for any $s>1$ and $z \in \Omega$, we have that
\begin{equation}
  \calM^\sharp_\Omega\left[|\GRAD \bu|\right](z) \lesssim \calM\left[ |\bef|^s \right](z)^{1/s} + \calM\left[|g|^s\right](z)^{1/s},
  \label{eq:bound_osc_velocity}
\end{equation}
where the hidden constant is independent of $\bef$, $g$, and $z$.
\end{lemma}
\begin{proof}
Let $Q$ be a cube \EO{with sides parallel to the coordinate axes and} center in $z$ such that $\frac32Q^\star \subset \Omega$, where $Q^\star = 2Q$. Extend $\bef$ and $g$ to zero outside $\Omega$ and decompose $\bef = \bef_1 + \bef_2$, with $\bef_1 = \bef\chi_{Q^\star}$, and $g= g_1 + g_2$ with $g_i$, $i=1,2$, as in Lemma~\ref{lem:Bogovskiidecompose} but with $Q$ replaced by $Q^\star$. Let now $\bu^i$ be the velocity component of the solution to \eqref{eq:Stokes} with data $(-\DIV \bef_i,g_i)$. 
\EO{To obtain \eqref{eq:bound_osc_velocity},} it thus suffices 
to bound the oscillation of $\partial_{x_k} \bu_j$ for all $k,j \in \{1,2,3\}$. Fix $j,k \in \{ 1,2,3\}$ and set $v = \bu_j$. With this notation, for $i \in \{1,2\}$, we have $\partial_{x_k} v^i = \partial_{x_k} \bu_j^i$.

To estimate $\calM_\Omega^\sharp[\partial_{x_k} v ]$ we follow \eqref{eq:average_does_not_matter} and bound the average of the difference between $\partial_{x_k} v$ and any constant. Thus, we have
\begin{align*}
  \fint_Q |\partial_{x_k} v(x) - \partial_{x_k} v^2(z) |\diff x &\leq \fint_Q |\partial_{x_k} v^1(x)|\diff x + \fint_Q |\partial_{x_k} v^2(x) - \partial_{x_k} v^2(z) |\diff x \\
  & =: \textup{I} + \textup{II}.
\end{align*}
We bound each of the terms separately.

First, by H\"older's inequality, for any $s>1$, we have
\[
  \textup{I} \leq \left( \fint_Q |\partial_{x_k} v^1(x)|^s \diff x \right)^{\frac{1}{s}} \AJS{\leq \frac1{|Q|^{1/s}} \| \partial_{x_k} v^1 \|_{L^s(\Omega)} \lesssim \frac1{|Q|^{1/s}} \| \vare(\bu^1) \|_{\bL^s(\Omega)},}
\]
\AJS{where, in the last step, we used Korn's inequality \eqref{eq:WeightedKorn} with $q=s$ and $\omega \equiv 1$. We now apply the unweighted estimate given in Theorem~\ref{thm:WPLq} to obtain 
\[
  \textup{I} \lesssim \frac1{|Q|^{1/s}} \left( \| \bef_1 \|_{\bL^s(Q^\star)} + \| g_1 \|_{L^s(\frac32 Q^\star)} \right).
\]
\EO{To obtain the previous estimate, we have also used that $\supp g_1 \subset \frac32 Q^\star$ and that $\bef_1$ vanishes outside $Q^\star$. Now, in view of the fact that} $|Q| \approx |Q^\star|$, $\bef_1 = \bef \chi_{Q^\star}$, and that the norm estimate on $g_1$ is \emph{local}, we finally arrive at
\[
  \textup{I} \lesssim \left( \fint_{Q^\star} |\bef(x)|^s \diff x \right)^{\frac{1}{s}} + \left( \fint_{\frac32 Q^\star} |g(x)|^s \diff x \right)^{\frac{1}{s}} \leq \calM\left[ |\bef|^s \right](z)^{\frac{1}{s}} + \calM\left[|g|^s\right](z)^{\frac{1}{s}}.
\]}

To bound $\textup{II}$ we observe that, since $x,z \notin \supp \bef_2 \cup \supp g_2$, it is legitimate to differentiate the pointwise representation of $v^2$ given in \eqref{eq:newptwise}. Consequently, we get
\begin{align*}
  \textup{II} &\leq \frac1\mu \fint_Q \int_{\Omega \cap (Q^\star)^c} \left|\partial_{x_k} \GRAD_y\bG_j(\EO{y,x}) - \partial_{x_k} \GRAD_y\bG_j(\EO{y,z})\right| |\bef_2(y)| \diff y \diff x \\ 
    &+ \fint_Q \int_{\Omega \cap (Q^\star)^c} \left|\partial_{x_k} \lambda_j(\EO{y,x}) - \partial_{x_k} \lambda_j(\EO{y,z})\right| |g_2(y)| \diff y \diff x =: \textup{II}_1 + \textup{II}_2.
\end{align*}
Now, \AJS{since $\bef_2 = \bef$ in $\Omega \cap (Q^\star)^c$, we can use} Theorem~\ref{thm:HolderGreen} with $|\alpha|=|\beta|=1$ and $i,j \in \{1,2,3\}$ \AJS{to} conclude that
\begin{align*}
  \textup{II}_1 &\lesssim \fint_Q \int_{(Q^\star)^c}|x-z|^\sigma \left( |x-y|^{-3-\sigma} + |z-y|^{-3-\sigma} \right) |\bef(y)|\diff y \diff x \\
    &\lesssim \frac{\ell(Q)^\sigma}{|Q|}\int_Q \int_{(Q^\star)^c} \frac{|\bef(y)|}{|z-y|^{3+\sigma}} \diff y \diff x \lesssim \calM\left[|\bef|\right](z),
\end{align*}
where, in the last two steps, we argued as in the proof of \cite[Lemma 2.4]{DDO:17}.

Similarly, we can use Theorem~\ref{thm:HolderGreen} with $|\alpha|=1$, $\beta = 0$, \EO{$i=4$, and $j \in \{1,2,3\}$} to assert that
\[
  \textup{II}_2 \lesssim \fint_Q \int_{(Q^\star)^c}|x-z|^\sigma \left( |x-y|^{-3-\sigma} + |z-y|^{-3-\sigma} \right) |\AJS{g_2}(y)|\diff y \diff x \lesssim \calM\left[|\AJS{g_2}|\right](z),
\]
where the argument, once again, follows the proof of \cite[Lemma 2.4]{DDO:17}. \EO{Let $R$ be a cube with center in $z$ and $x \in R$.} \AJS{We follow the notation of Lemma~\ref{lem:Bogovskiidecompose} and observe that we have the pointwise estimate
\[
  |g_2(x)| \leq |g(x)| + \frac{\chi_{A^\star}(x)}{|A^\star|}\int_{D^\star} |g(\xi)| \diff \xi \lesssim |g(x)| + \left(\fint_{D^\star} |g(\xi)| \diff \xi \right) \chi_{A^\star}(x),
\]
where we used that $|A^\star|\approx |D^\star|$. The sublinearity of the maximal function then implies that
\begin{align*}
  \textup{II}_2 &\lesssim \calM\left[|g|\right](z) + \left(\fint_{D^\star} |g(\xi)| \diff \xi \right) \calM\left[\chi_{A^\star}\right](z) \leq \calM\left[|g|\right](z) \left( 1+ \calM\left[\chi_{A^\star}\right](z) \right) \\
    &\lesssim \calM\left[|g|\right](z),
\end{align*}
where in the last step we used that the maximal function is nonexpansive in $L^\infty(\R^3)$} \EO{i.e., $\| \calM f \|_{L^{\infty} (\mathbb{R}^3)} \leq \| f \|_{L^{\infty}(\mathbb{R}^3)}$ for every $f \in L^1_{\mathrm{loc}}(\mathbb{R}^3)$ \cite[\EO{Section 2.1}, page 86]{MR3243734}.}

Conclude by noticing that for every $s>1$, by H\"older's inequality, we have that $\calM[|w|](z) \leq \calM[|w|^s](z)^{1/s}$.
\end{proof}

The weighted a priori estimate of the velocity component of the solution to \eqref{eq:Stokes} is the content of the following result.

\begin{theorem}[velocity estimate]
\label{thm:velest}
Let $\Omega \subset \R^3$ be a convex polyhedron, $q \in (1,\infty)$, $\omega \in A_q$, $\bef \in \bL^q(\omega,\Omega)$, and $g \in \cL^q(\omega,\Omega)$. Let $\bu$ be the velocity component of the solution of \eqref{eq:Stokes} with $\bF = -\DIV \bef$. Then, we have that
\[
  \| \GRAD \bu \|_{\bL^q(\omega,\Omega)} \lesssim \| \bef \|_{\bL^q(\omega,\Omega)} + \| g \|_{L^q(\omega,\Omega)},
\]
where the hidden constant is independent of $\bef$ and $g$ and depends on $\omega$ only through $[\omega]_{A_q}$.
\end{theorem}
\begin{proof}
We argue as in \cite[Theorem 2.5]{DDO:17}. Let $(\GRAD \bu)_\Omega = \fint_\Omega \GRAD \bu \diff x$. Then,
\[
  \| \GRAD \bu \|_{\bL^q(\omega,\Omega)} \leq \| \GRAD \bu - (\GRAD \bu)_\Omega\|_{\bL^q(\omega,\Omega)} + \| (\GRAD \bu)_\Omega \|_{\bL^q(\omega,\Omega)} =: \textup{I} + \textup{II}.
\]

\EO{To bound the term $\textup{I}$, we first utilize the estimate in \cite[Theorem 5.15]{MR2643399} and then} the bound from Lemma~\ref{lem:lem24Duran} to obtain
\[
  \textup{I} 
\lesssim
  \left\| \calM_\Omega^\sharp\left[|\GRAD \bu | \right] \right\|_{\bL^q(\omega,\Omega)} \lesssim \| \calM[|\bef|^s]^{1/s} \|_{\bL^q(\omega,\Omega)} + \| \calM[|g|^s]^{1/s} \|_{L^q(\omega,\Omega)}.
\]
Now, using the so-called open ended property of Muckenhoupt weights \cite[Corollary 1.2.17]{MR1774162}, we have that there is $s>1$ such that $\omega \in A_q$ implies $\omega \in A_{q/s}$. Thus, using the boundedness of $\calM$ over Muckenhoupt weighted spaces we obtain
\[
  \| \calM[|g|^s]^{1/s} \|_{L^q(\omega,\Omega)} = \| \calM[|g|^s] \|_{L^{q/s}(\omega,\Omega)}^{1/s} \lesssim \| |g|^s \|_{L^{q/s}(\omega,\Omega)}^{1/s} = 
  \| g \|_{L^q(\omega,\Omega)};
\]
a similar estimate holds for $\| \calM[|\bef|^s]^{1/s} \|_{\bL^q(\omega,\Omega)}$. In conclusion
\[
  \textup{I}\lesssim \| \bef \|_{\bL^q(\omega,\Omega)} + \| g \|_{L^q(\omega,\Omega)}.
\]

To bound $\textup{II}$ we use the unweighted estimate of Theorem~\ref{thm:WPLq} to obtain
\[
  |(\GRAD \bu)_\Omega| \leq \left( \fint_\Omega |\GRAD \bu|^s \diff x \right)^{1/s} \lesssim \left( \fint_{\Omega} |\bef|^s \diff x \right)^{1/s} + \left( \fint_\Omega |g|^s \diff x \right)^{1/s}.
\]
Now, by H\"older's inequality,
\[
  \left( \fint_\Omega |g|^s \diff x \right)^{1/s} \leq \left( \fint_\Omega |g|^q \omega \diff x \right)^{1/q} \left( \fint_\Omega \omega^{s/(s-q)} \diff x\right)^{1/s-1/q},
\]
with a similar estimate for $\left( \fint_{\Omega} |\bef|^s \diff x \right)^{1/s}$. We can thus obtain the estimate 
\begin{align*}
  \int_\Omega \omega |(\GRAD \bu)_\Omega|^q \diff x &\lesssim \left( \fint_\Omega \omega^{s/(s-q)} \diff x\right)^{q/s-1} \left(\fint_\Omega \omega \diff x\right)\left( \| \bef \|_{\bL^q(\omega,\Omega)}^q + \| g\|_{L^q(\omega,\Omega)}^q \right) \\
  & \lesssim [\omega]_{A_{q/s}} \left( \| \bef \|_{\bL^q(\omega,\Omega)}^q + \| g\|_{L^q(\omega,\Omega)}^q \right).
\end{align*}

The theorem is thus proved.
\end{proof}

We now obtain an a priori estimate on the pressure.

\begin{corollary}[pressure estimate]
\label{col:presest}
In the setting of Theorem~\ref{thm:velest}, if $\pe$ is the pressure component of the solution to \eqref{eq:Stokes}, then we have
\[
  \| \pe \|_{L^q(\omega,\Omega)} \lesssim \| \bef \|_{\bL^q(\omega,\Omega)} + \| g \|_{L^q(\omega,\Omega)},
\]
where the hidden constant is independent of $\bef$ and $g$ and depends on $\omega$ only through $[\omega]_{A_q}$.
\end{corollary}
\begin{proof}
The proof is a simple application of the the inf-sup condition \eqref{eq:infsuppres}. Indeed, using that $( \bu,\pe)$ solves \eqref{eq:Stokes} and the conclusion of Theorem~\ref{thm:velest} we obtain
\begin{align*}
  \| \pe \|_{L^q(\omega,\Omega)} &\lesssim \sup_{\boldsymbol0 \neq \bv \in \bW^{1,q'}(\omega',\Omega)} \frac1{\|\GRAD \bv\|_{\bL^{q'}(\omega',\Omega)}} \int_\Omega \left( \bef:\GRAD \bv   - \AJS{2} \mu \vare(\bu):\vare(\bv) \right)\diff x \\ &\lesssim \| \bef \|_{\bL^q(\omega,\Omega)} + \| g \|_{L^q(\omega,\Omega)},
\end{align*}
as we intended to show.
\end{proof}

%

We conclude with a corollary regarding the stability of the Stokes projection on weighted spaces. In doing so, we will remove some of the assumptions used in \cite[Theorem 4.1]{DOS:19}. Namely, we no longer have a lower bound on the integrability index and we do not require good behavior of the weight near the boundary.

\begin{corollary}[stability]
\label{cor:stabStokesProj}
Let $\Omega \subset \R^3$ be a convex polyhedron. Assume that either:
\begin{enumerate}[$\bullet$]
  \item $q\in [2,\infty)$ and $\omega \in A_{q/2}$,
  \item $q \in (1,2]$ and $\omega' \in A_{q'/2}$.
\end{enumerate}
Then, the Stokes projection defined in \eqref{eq:defofStokesprojection} is stable on $\bW^{1,q}_0(\omega,\Omega) \times \cL^q(\omega,\Omega)$.
\end{corollary}
\begin{proof}
The assertion for $q=2$ is Theorem~\ref{thm:stabclassic}. Now, according to \cite[page 142]{MR1800316}, the following variant of the extrapolation theorem \cite[Theorem 7.8]{MR1800316} can be derived: given $s\geq 1$, if $T$ is a  bounded operator on $L^r(\rho,\Omega)$ for all $\rho \in A_{r/s}$, then for $q>s$ it is bounded on $L^q(\varpi,\Omega)$ for all $\varpi \in A_{q/s}$. Use this result with $r=s=2$ to conclude the stability for $q>2$.

For the second case, repeat \AJS{the} duality argument given in the proof of \cite[Theorem 4.1]{DOS:19}. But, since we are in a convex polyhedron, we use the well-posedness of Theorem~\ref{thm:velest}, so that there is no restriction on $q'>2$, nor it is required that the weight behaves nicely close to the boundary. Conclude using the just obtained stability of the Stokes projection for $q'>2$ and $\omega' \in A_{q'/2}$.
\end{proof}

\section{A class of non-Newtonian fluids under singular forcing}
\label{sec:Bulicek2}

Recently, a class of non-Newtonian fluids was studied in reference \cite{MR3582412}. \EO{Such a class} fits \eqref{eq:nlStokes} with the following data and assumptions. 

\subsection{Assumptions}
\label{sec:Bulicekassumptions}

Let $q \in (1,\infty)$ and $\omega \in A_q$. Assume that $g \in \cL^q(\omega,\Omega)$, $\bef \in \bL^q(\omega,\Omega)$, and that the nonlinear stress tensor $\polS$ satisfy \EO{the following assumptions}:
\begin{enumerate}[$\bullet$]
  \item \emph{Measurability and continuity}: The mapping $\polS : \Omega \times \R^{3\times3} \to \R^{3\times3}$ is Carath\'eodory.
  
  \item \emph{Coercivity and growth}: For all $\bQ \in \R^{3\times3}$ and every $x \in \Omega$ we have
  \[
    |\bQ^s|^2 -1 \lesssim \polS(x,\bQ^s) : \bQ, \quad |\polS(x,\bQ^s)| \lesssim |\bQ| + 1, 
  \]
  where $\bQ^s = \tfrac12(\bQ+\bQ^\intercal)$.
  
  \item \emph{Linearity at infinity}: There is a positive number $\mu$ such that for all $\bQ \in \R^{3\times3}$ and \EO{uniformly in} $x$ we have
  \[
    \lim_{|\bQ^s|\to \infty} \frac{|\polS(x,\bQ^s) - \AJS{2}\mu \bQ^s|}{|\bQ^s|} = 0.
  \]
  
  \item \emph{Strict monotonicity and strong asymptotic Uhlenbeck condition:} For all $\bQ,\bP \in \R^{3\times 3}$, with $\bQ^s \neq \bP^s$, and \EO{uniformly in} $x$
  \[
    \left( \polS(x,\bQ^s) - \polS(x,\bP^s) \right) : (\bQ-\bP) > 0
  \]
  and
  \[
    \lim_{|\bQ^s|\to \infty} \left| \frac{\partial\polS(x,\bQ^s)}{\partial \bQ^s} - \AJS{2}\mu \bI \right| = 0.
  \]
\end{enumerate}

\EO{An instance of a nonlinear stress tensor $\mathbb{S}$ satisfying all the previous assumptions can be found in \cite[(1.2)]{MR3582412}.}

\subsection{Well-posedness}
\label{sec:Bulicek}

Under the assumption that $\Omega$ has \EO{a} $C^1$ boundary, the authors of \cite{MR3582412} show existence, uniqueness, and a stability estimate for the solution to \eqref{eq:nlStokes} with the hypotheses given above; see \EO{\cite[Theorem 1.5]{MR3582412}}. Let us, with the help of the results of Section~\ref{sec:Apweights}, extend this theory to convex polyhedra.

\begin{theorem}[well-posedness]
\label{thm:Bulicek}
Let $\Omega \subset \R^3$ be a convex polyhedron. Assume that $q\in(1,\infty)$ and $\omega \in A_q$. If $\bef \in \bL^q(\omega,\Omega)$, $g \in \cL^q(\omega,\Omega)$, and the stress tensor satisfies all the aforementioned conditions, \EO{then} there is a unique pair $(\bu,\pe) \in \bW^{1,q}_0(\omega,\Omega) \times \cL^q(\omega,\Omega)$ that solves \eqref{eq:nlStokes} and satisfies the estimate
\[
  \| \GRAD \bu \|_{\bL^q(\omega,\Omega)} + \| \pe \|_{L^q(\omega,\Omega)} \lesssim 1 + \| \bef \|_{\bL^q(\omega,\Omega)} + \| g \|_{L^q(\omega,\Omega)},
\]
where the hidden constant only depends on $q$, $[\omega]_{A_q}$, and the constants involved in the properties that $\polS$ satisfies.
\end{theorem}
\begin{proof}
The proof follows after minor modifications of the proof of \cite[Theorem 1.5]{MR3582412}. We will only indicate the main steps that need to be changed.

First, we consider \eqref{eq:Stokes} with $q \in (1,\infty)$, $\omega \in A_q$, $(\bef,g) \in \bL^q(\omega,\Omega) \times \cL^q(\omega,\Omega)$, $\bF = -\DIV \bef$, and $\mu$ as in the assumptions for $\polS$. Thus, owing to the results of Theorem~\ref{thm:velest} and Corollary~\ref{col:presest}, this problem is well-posed provided $\Omega \subset \R^3$ is a convex polyhedron.

Next, as in \cite[Section 3.2]{MR3582412}, we find, for $q=2$, suitable a priori estimates for solutions of \eqref{eq:nlStokes}. The first idea behind the argument is to \EO{assume that $\GRAD \bu \in \bL^s(\Omega)$ and $\pe \in L^s(\Omega)$, with $s \in (1,2]$, and} approximate the weight $\omega \in A_2$ by $\omega_j$ \EO{in such a way that} $\GRAD \bu \in \bL^2(\omega_j,\Omega)$ and $\pe \in \cL^2(\omega_j,\Omega)$. This is accomplished by defining
\begin{equation}
\label{eq:defofomegaj}
\tilde\omega_1 = \calM[|\GRAD \bu|]^{s-2},
\quad
\tilde\omega_2 = \calM[|\pe|]^{s-2}, 
\quad
\tilde\omega_3 = \min\{ \tilde\omega_1,\tilde\omega_2\}.
\end{equation}
\EO{Since $s \in (1,2]$, we have that $\tilde\omega_1, \tilde \omega_2 \in A_2$ \cite[Lemma 2.4]{MR3582412}, which implies that $\tilde\omega_3 \in A_2$ \cite[estimate (2.4)]{MR3582412}. Since, by assumption,} $\GRAD \bu \in \bL^s(\Omega)$ and $\pe \in L^s(\Omega)$, for  $s \in (1,2]$, \cite[estimate (2.6)]{MR3582412} yields $\GRAD \bu \in \bL^2(\tilde \omega_1,\Omega)$ and $\pe \in L^2(\tilde \omega_2,\Omega)$. Let us define, for $j \in \polN$, the weight $\omega_j := \min\{ j \tilde\omega_3,\omega\}$. Notice that $\omega_j \in A_2$ and that
\[
  \GRAD \bu \in \bL^2(\omega_j,\Omega), \quad \pe \in L^2(\omega_j,\Omega).
\]

\EO{We now rewrite \eqref{eq:nlStokes} as a linear problem: Find $(\bu,\pe) \in \bW^{1,2}_0(\omega_j,\Omega) \times \cL^2(\omega_j,\Omega)$ such that}
\begin{align*}
  \int_\Omega \left( \AJS{2}\mu \vare(\bu) : \vare(\bv) - \pe\DIV \bv \right) \diff x &= \int_\Omega \left( \bef + \AJS{2}\mu\vare(\bu) - \polS(x,\vare(\bu)) \right):\GRAD \bv \diff x \\
  \int_\Omega \DIV \bu r \diff x = \int_\Omega g r \diff x,
\end{align*}
for every $(\bv,r) \in \bW^{1,2}_0(\omega_j^{-1},\Omega) \times \cL^2(\omega_j^{-1},\Omega)$.
The previous arguments, in conjunction with the estimates of Theorem~\ref{thm:velest} and Corollary~\ref{col:presest}, then imply that
\begin{multline}
\label{eq:keyaprioriest}
  \| \GRAD \bu \|_{\bL^2(\omega_j,\Omega)}^2 + \|\pe \|_{L^2(\omega_j,\Omega)}^2 \lesssim 1 + \| \bef \|_{\bL^2(\omega_j,\Omega)}^2 + \| g \|_{L^2(\omega_j,\Omega)}^2
  \\ + \int_{\{|\vare(\bu)|\geq m\}} \omega_j \frac{|\polS(x,\vare(\bu))-\AJS{2}\mu \vare(\bu)|^2}{|\vare(\bu)|^2} |\vare(\bu)|^2 \diff x,    
\end{multline}
where the \EO{hidden} constant depends on $\omega$, $m$, and the properties of $\polS$. By properly choosing $m$ and using that $\polS$ has linear growth at infinity, we conclude that
\[
  \| \GRAD \bu \|_{\bL^2(\omega_j,\Omega)}^2 + \|\pe \|_{L^2(\omega_j,\Omega)}^2 \lesssim 1 + \| \bef \|_{\bL^2(\omega_j,\Omega)}^2 + \| g \|_{L^2(\omega_j,\Omega)}^2,
\]
which is uniform in $j$. Passing to the limit and cleaning up the proof, the desired a priori estimate is obtained for the case $q=2$; see \cite[Section 3.2]{MR3582412} for details. For $q\neq2$ it suffices to invoke the extrapolation theorem provided in \cite[Theorem 7.8]{MR1800316}.

The rest of the proof, \ie existence and uniqueness follows \emph{verbatim} the results of \cite{MR3582412}. It is only worth noticing that here, once again it is necessary to use, via Theorem~\ref{thm:WPLq}, the fact that we are in a convex polyhedron.
\end{proof}

\begin{remark}[novelty]
Since the proof of Theorem~\ref{thm:Bulicek} follows \cite{MR3582412} one must wonder what is the novelty here. The main difference lies in the fact that estimate \eqref{eq:keyaprioriest}, over convex polyhedra, can only be obtained by invoking Theorem~\ref{thm:velest} and Corollary~\ref{col:presest}. Indeed, the results of \cite{MR3582412} require a localization argument for the linear problem that would not work in general over polyhedra. On the other hand, the well-posedness results of \cite{OS:17infsup} require that the weights $\omega_j \in A_2(\Omega)$ which seems impossible to guarantee, since they depend on the solution itself; see \EO{the involved definitions in} \eqref{eq:defofomegaj}.
\end{remark}

\subsection{Discretization}
\label{sub:DiscreteBulicek}

Let us investigate the convergence properties of finite element approximations. We will operate under the setting described in Corollary~\ref{cor:stabStokesProj} and seek for a pair $(\bu_h,\pe_h) \in \bX_h \times M_h$ that satisfies
\begin{equation}
\label{eq:discrBulicek}
   \begin{dcases}
    \int_\Omega \polS(x,\vare(\bu_h)) : \EO{ \nabla \bv_h } \diff x - \int_\Omega \pe_h \DIV \bv_h \diff x = \int_\Omega \bef :\GRAD \bv_h \diff x & \forall \bv_h \in \bX_h, \\
    \int_\Omega \DIV \bu_h r_h \diff x = 0 & \forall r_h \in M_h,
  \end{dcases}
\end{equation}
where $\polS$ is assumed to satisfy the conditions described in section~\ref{sec:Bulicekassumptions}. Our main result regarding the convergence of discretizations is the following.

\begin{theorem}[convergence]
\label{thm:convergenceBulicek}
Assume that either $q \in [2,\infty)$ and $\omega \in A_{q/2}$ or $q \in (1,2]$ and $\omega' \in A_{q'/2}$. Let $\bef \in \bL^q(\omega,\Omega)$. For every $h>0$, problem \eqref{eq:discrBulicek} admits 
\EO{a unique} solution which satisfies the estimate
\[
  \| \GRAD \bu_h \|_{\bL^q(\omega,\Omega)} + \| \pe_h \|_{L^q(\omega,\Omega)} \lesssim 1 + \| \bef \|_{\bL^q(\omega,\Omega)},
\]
where the hidden constant does not depend on $h$. Moreover, as $h\to 0$, there is a subsequence of $\{\bu_h\}_{h>0}$ that converges weakly, in $\bW^{1,q}_0(\omega,\Omega)$ to $\bu$, the solution of problem \eqref{eq:nlStokes} with $g=0$.
\end{theorem}
\begin{proof}
Existence \EO{and uniqueness of} solutions follows from standard monotone operator theory \cite[Chapter 2]{MR3014456}. Let us now provide the claimed a priori bound with an argument similar to that of Theorem~\ref{thm:Bulicek}. We see that the pair $(\bu_h,\pe_h)$ is such that
\begin{align*}
  \int_\Omega \left( \AJS{2}\mu \vare(\bu_h) : \vare(\bv_h) - \pe_h\DIV \bv_h \right) \diff x &= \int_\Omega \left( \bef + \AJS{2}\mu\vare(\bu_h) - \polS(x,\vare(\bu_h)) \right):\GRAD \bv_h \diff x, \\
  \int_\Omega \DIV \bu_h r_h \diff x &= 0,
\end{align*}
for every $(\bv_h,r_h) \in \bX_h \times M_h$. The stability of the Stokes projection on weighted spaces, proved in Corollary~\ref{cor:stabStokesProj}, implies that \EO{the pair $(\bu_h,\pe_h)$} satisfies an estimate similar to \eqref{eq:keyaprioriest}. Namely,
\begin{multline}
\label{eq:apriori_estimate_h}
  \| \GRAD \bu_h \|_{\bL^q(\omega,\Omega)} + \|\pe_h \|_{L^q(\omega,\Omega)} \lesssim 1 + \| \bef \|_{\bL^q(\omega,\Omega)}
  \\ + \left( \int_{\{|\vare(\bu_h)|\geq m\}} \omega \frac{|\polS(x,\vare(\bu_h))-\AJS{2}\mu \vare(\bu_h)|^q}{|\vare(\bu_h)|^q} |\vare(\bu_h)|^q \diff x \right)^{\frac{1}{q}}.
\end{multline}
Once again, \EO{by properly choosing $m$ and utilizing the growth properties of $\polS$, we} conclude the desired estimate; the hidden constant being independent of $h$.

The a priori estimate allows us to extract a (not relabeled) subsequence $\{\bu_h\}_{h>0}$ such that  $\bu_h \rightharpoonup \tilde\bu$ in $\bW^{1,q}_0(\omega,\Omega)$.
It remains then to show that $ \tilde\bu$ solves \eqref{eq:nlStokes}. To see this,
notice, first of all, that $\tilde \bu$ must be solenoidal. Let $\bv \in \bC_0^\infty(\Omega)$ be solenoidal.  
Set $\bv_h = \Fortin \bv$, where $\Fortin$ is the Fortin operator,  as a test function on \eqref{eq:discrBulicek} to obtain
\[
  \int_\Omega \polS(x,\vare(\bu_h)) : \EO{\nabla} \Fortin \bv \diff x = \int_\Omega \bef :\GRAD \Fortin\bv \diff x.
\]
Notice that $\Fortin \bv \to \bv$ in $\bW^{1,q'}_0(\omega',\Omega)$. On the other hand, the properties of $\polS$ imply that there is $\polD \in \bL^q(\omega,\Omega)$ for which $\polS(x,\vare(\bu_h)) \rightharpoonup \polD$, so that, passing to the limit $h \to 0$, the previous identity implies that
\[
  \int_\Omega \polD : \EO{\nabla} \bv \diff x = \int_\Omega \bef :\GRAD \bv \diff x \quad \forall \bv \in \bW^{1,q'}_0(\omega',\Omega), \ \DIV \bv = 0.
\]

We would like to conclude with a variant of Minty's trick \cite[Lemma 2.13]{MR3014456}. However, as described in \cite[section 5.3]{MR3582412}, this is not possible as $\tilde \bu$ is not an admissible test function. However, the same reference has shown how to deal with this. The important point to note here is that the necessary technical steps developed there, \cite[Theorems 1.9 and 1.10]{MR3582412}, do not assume smoothness on the domain $\Omega$. In conclusion, $\polD = \polS(x,\vare(\tilde\bu))$ and, by uniqueness, $\tilde \bu = \bu$.
\end{proof}

\section{The Smagorinski model of turbulence}
\label{sec:Rappazz}

In the subgrid modeling of turbulence, one of the first proposed models was the so-called Smagorinski model \cite{Smagorinski}, which is nothing but a special case of the Lady\v{z}enskaja model of non-Newtonian fluid flow. Some history on this model is described in \cite{MR2471134,MR2119869}. One of the main issues regarding this model is that, as it has been observed \EO{in} \cite{MR3494304,Lesieur}, it tends to overdissipate the flow near the boundaries. Several refinements of this model have been proposed, \cite{MR1815221,MR1990952,MR3093471,MR2185509}, and here we wish to provide some analysis to one of them.

In \cite{MR3488119,MR3837915}, an analysis of a simplified version of the so-called Smagorinski model of turbulence was studied. This problem takes the form of \eqref{eq:nlStokes} but the stress tensor, in this case, is
\begin{equation}
\label{eq:Smagorinski}
  \polS(x,\bQ^s) = \AJS{2}\left( \mu + \mu_{NL} \dist(x,\partial\Omega)^\alpha |\bQ^s|\right) \bQ^s, \qquad \alpha \in [0,2),
\end{equation}
\EO{where $\mu>0$ corresponds to a laminar viscosity and $\mu_{NL}$ denotes a positive constant.} The idea here is that the additional, nonlinear viscosity, vanishes near the walls. The authors of \cite{MR3488119,MR3837915} show the existence and uniqueness of a velocity field $\bu$ and the existence of a pressure field $p$ in
\[
\bW^{1,2}_0(\Omega) \cap \bW^{1,3}(\dist(\cdot,\partial\Omega)^\alpha,\Omega), 
\quad
\cL^2(\Omega) \AJS{+} \cL^{3/2}(\dist(\cdot, \partial \Omega)^{-\alpha/2},\Omega),
\]
respectively; the pressure being, in general not unique \EO{\cite[Theorem 1]{MR3837915}. Uniqueness can be guaranteed under further restrictions on $\alpha$: there exists $\alpha_0 \leq 1/2$ such that for every $\alpha < \alpha_0$ problem \eqref{eq:nlStokes} admits a unique solution $(\bu,p) \in  \bW^{1,3}(\dist(\cdot,\partial\Omega)^\alpha,\Omega) \times \cL^{3/2}(\dist(\cdot, \partial \Omega)^{-\alpha/2},\Omega)$ \cite[Theorem 2]{MR3837915}.} It is important to note that, even for $\alpha = 0$, \eqref{eq:Smagorinski} does not fit in the framework developed in Section~\ref{sec:Bulicek2}. In particular, such a stress tensor is not linear at infinity. This \EO{may} serve as an explanation for the difference in the functional setting that needs to be adopted to analyze \eqref{eq:nlStokes} with the stress tensor \EO{$\polS$} in the form \eqref{eq:Smagorinski}.

\subsection{Analysis}
\label{sub:AnalysisgenSmago}

Let us, in the framework that we have developed so far, study a generalization of this problem and, in addition, provide an error estimate for finite element approximations. The specialization to the Smagorinski model \eqref{eq:Smagorinski} will be immediate. \EO{We will be interested in the following problem}: given $q \in (1,\infty)$, a weight $\omega \in A_q$, a forcing $\bef$, and $g=0$, find $(\bu,\pe)$ that solves \eqref{eq:nlStokes}, where the stress tensor is given by
\begin{equation}
\label{eq:genSmagorinskivisco}
  \polS(x,\bQ^s) = \AJS{2}\left(\mu + \omega(x) |\bQ^s|^{q-2}\right) \bQ^s.
\end{equation}

We now proceed with \EO{an analysis to} this problem. To accomplish this task, we first introduce some notation. For $q \in (1,\infty)$ and $\omega \in A_q$, we set
\[
  \bX_q(\omega) = \bW^{1,2}_0(\Omega) \cap \bW^{1,q}_0(\omega,\Omega), \qquad M_q(\omega) = \cL^2(\Omega) \AJS{+} \cL^{q'}(\omega',\Omega).
\]
Slight modifications of the arguments in \cite{MR3488119,MR3837915} will give existence of solutions.

\begin{theorem}[existence]
\label{thm:Rappaz}
Let $\Omega \subset \R^3$ be a convex polyhedron, $q \in (1,\infty)$, $\omega \in A_q$, and $\bef \in \bL^2(\Omega) \AJS{+} \bL^{q'}(\omega',\Omega)$.
Then, there is 
$
  (\bu,\pe) \in \bX_q(\omega) \times M_q(\omega)
$
such that
\[
  \begin{dcases}
    \int_\Omega \polS(x,\vare(\bu)) :\vare(\bv) \diff x - \int_\Omega \pe \DIV \bv \diff x = \int_\Omega \bef: \GRAD \bv \diff x  & \forall \bv \in \bX_q(\omega), \\
    \int_\Omega \DIV \bu r \diff x  = 0  & \forall r \in M_q(\omega),
  \end{dcases}
\]
where $\polS$ is given by \eqref{eq:genSmagorinskivisco}. The velocity component of this pair is unique.
\end{theorem}
\begin{proof}
We essentially follow \cite[Sections 2.2 and 2.3]{MR3837915}. Define the functional
\[
  \calJ(\bv) = \int_\Omega \polA(x,\vare(\bv)) \diff x - \int_\Omega \bef : \GRAD \bv \diff x,
\]
with
\[
  \polA(x,\bQ^s) = \AJS{\mu} |\bQ^s|^2 + \frac{\AJS{2}}q\omega(x)|\bQ^s|^q.
\]
We wish to minimize $\EO{\calJ}$ over $\bX_{q,\DIV}(\omega):=\{ \bw \in \bX_q(\omega) : \DIV \bw = 0\}$. Observe that $\calJ$ is G\^{a}teaux differentiable and
\[
  \calJ'(\bu) \bv =  \int_\Omega \polS(x,\vare(\bu)) :\vare(\bv) \diff x - \int_\Omega \bef : \GRAD \bv \diff x \quad \forall \bv \in \bX_{q,\DIV}(\omega).
\]
\AJS{In addition, $\calJ$ is strictly convex and continuous on $\bX_{q,\DIV}(\omega)$; \EO{see \cite[Lemma 7]{MR3837915} and \cite[Lemma 6]{MR3837915}, respectively.} Moreover, owing to Korn's inequality \eqref{eq:WeightedKorn} and its unweighted version, \EO{\ie \eqref{eq:WeightedKorn} with $q=2$ and $\omega \equiv 1$}, we have
\[
  \frac{\calJ(\bv)}{\| \bv \|_{\bX_{q,\DIV}(\omega)}} = \frac{ \mu \| \vare(\bv) \|_{\bL^2(\Omega)}^2 + \frac2q \| \vare(\bv) \|_{\bL^q(\omega,\Omega)}^q}{\| \GRAD\bv \|_{\bL^2(\Omega)} +  \| \GRAD \bv \|_{\bL^q(\omega,\Omega)}} \gtrsim \frac{ \| \GRAD \bv \|_{\bL^2(\Omega)}^2 + \| \GRAD \bv \|_{\bL^q(\omega,\Omega)}^q}{\| \GRAD\bv \|_{\bL^2(\Omega)} +  \| \GRAD \bv \|_{\bL^q(\omega,\Omega)}}.
\]
We can thus obtain that
\[
 \| \bv \|_{\bX_{q,\DIV}(\omega)} \to \infty
 \EO{\implies}
   \frac{\calJ(\bv)}{\| \bv \|_{\bX_{q,\DIV}(\omega)}} \to \infty.
\]
In other words, $\calJ$ is coercive on $\bX_{q,\DIV}(\omega)$}. Consequently, direct methods provide the existence and uniqueness of a minimizer $\bu \in \bX_{q,\DIV}(\omega)$ which, in addition, satisfies
\[
  \int_\Omega \polS(x,\vare(\bu)) :\vare(\bv) \diff x = \int_\Omega \bef: \GRAD \bv \diff x \quad \forall \bv \in \bX_{q,\DIV}(\omega).
\]

To find the pressure we apply the inf-sup condition \eqref{eq:infsuppres} twice. First with $q=2$ and $\omega\equiv1$, to find that there is $\pe_1 \in \cL^2(\Omega)$ and another time with $q \in (1,\infty)$ and $\omega \in A_q$, to find $\pe_2 \in \cL^{q'}(\omega',\Omega)$, so that $\pe = \pe_1 + \pe_2 \in M_q(\omega)$ is such that
\[
  \int_\Omega \pe \DIV \bv \diff x = \int_\Omega \bef :\GRAD \bv \diff x - \int_\Omega \polS(x,\vare(\bu)) :\vare(\bv) \diff x \quad \forall \bv \in \bX_q(\omega).
\]
As announced, nothing can be said about the uniqueness of $\pe$; see the discussion in \cite[Remark 2]{MR3837915}. \EO{This concludes the proof.}
\end{proof}

\subsection{Discretization}
\label{sub:DiscrSmago}

We now proceed with a finite element approximation of problem \eqref{eq:nlStokes}, where the stress tensor $\polS$ is given by \eqref{eq:genSmagorinskivisco}. We will seek for $(\bu_h,\pe_h) \in \bX_h \times M_h$ such that
\begin{equation}
\label{eq:discrSmago}
   \begin{dcases}
    \int_\Omega \polS(x,\vare(\bu_h)) :\vare(\bv_h) \diff x - \int_\Omega \pe_h \DIV \bv_h \diff x = \int_\Omega \bef :\GRAD \bv_h \diff x & \forall \bv_h \in \bX_h, \\
    \int_\Omega \DIV \bu_h r_h \diff x = 0 & \forall r_h \in M_h.
  \end{dcases}
\end{equation}
Existence of \EO{a} solution follows along the lines of Theorem~\ref{thm:Rappaz}. Let us now provide \EO{an a priori error estimate}.

\begin{theorem}[error estimate]
\label{thm:errestSmago}
Let $q \in [2,\infty)$ \EO{and let} $(\bu,\pe) \in \bX_q(\omega) \times M_q(\omega)$ \EO{be a solution to} the generalized Smagorinski problem described by \eqref{eq:nlStokes} with the stress tensor \eqref{eq:genSmagorinskivisco}. Let $(\bu_h,\pe_h) \in \bX_h \times M_h$ denote \EO{a} solution to \eqref{eq:discrSmago}. If $\omega \in A_{q/2}$, then we have the following \EO{a priori error estimate:}
\begin{multline*}
  \| \vare(\bu - \bu_h )\|_{\bL^2(\Omega)}^2 + \| \vare(\bu - \bu_h )\|_{\bL^q(\omega,\Omega)}^q
  \\
   \lesssim \AJS{ \| \vare(\bu - \bU_h) \|_{\bL^2(\Omega)}^2 + \| \vare(\bu - \bU_h) \|_{\bL^q(\omega,\Omega)}^{q/(q-1)}, }
\end{multline*}
where the hidden constant is independent of $\bu$, $\pe$, $h$, \EO{and $(\bU_h,P_h)$; the latter being the Stokes projection of $(\bu,\pe)$}.
\end{theorem}
\begin{proof}
We will follow, for instance, the derivation of the estimates of \cite[Section 4]{MR2991839}. Namely, by conformity, we have that, for all $\bv_h \in \bX_h$
\begin{multline*}
  \AJS{2}\mu \int_\Omega \vare(\bu - \bu_h) :\vare(\bv_h) \diff x+ \AJS{2}\int_\Omega \left( |\vare(\bu)|^{q-2} \vare(\bu) - |\vare(\bu_h)|^{q-2} \vare(\bu_h) \right): \vare(\bv_h) \omega \diff x \\ = \int_\Omega (\pe - \pe_h) \DIV \bv_h \diff x.
\end{multline*}
Denote now by $(\bU_h,P_h) \in \bX_h \times M_h$ the Stokes projection of $(\bu,\pe)$, as defined in \eqref{eq:defofStokesprojection}.
Setting $\bv_h = \bU_h - \bu_h$, \EO{the previous} identity reduces to
\begin{multline*}
  \AJS{2}\mu \int_\Omega \vare(\bu - \bu_h) :\vare(\bU_h - \bu_h) \diff x\\+ \AJS{2}\int_\Omega \left( |\vare(\bu)|^{q-2} \vare(\bu) - |\vare(\bu_h)|^{q-2} \vare(\bu_h) \right) :\vare(\bU_h - \bu_h) \omega \diff x  \\ = \AJS{2}\mu \int_\Omega \vare(\bu-\bU_h) :\vare( \bU_h - \bu_h) \diff x,
\end{multline*}
where we used the fact that $\bU_h - \bu_h$ is discretely solenoidal and that \EO{$(\bU_h,P_h)$} is the Stokes projection of \EO{$(\bu,\pe)$}. This identity can be used to derive
\begin{multline}
\label{eq:needforql2}
  \EO{2}\AJS{\mu} \| \vare(\bu - \bu_h )\|_{\bL^2(\Omega)}^2  \\ 
  +
  \AJS{2}\int_\Omega \left( |\vare(\bu)|^{q-2} \vare(\bu) - |\vare(\bu_h)|^{q-2} \vare(\bu_h) \right): \vare(\bu - \bu_h) \omega \diff x \\ 
= \EO{-2\mu\| \vare(\bu - \bU_h )\|_{\bL^2(\Omega)}^2 +  4\mu \int_{\Omega}\vare(\bu - \bu_h ):\vare(\bu - \bU_h ) \diff x}
\\
\EO{+ \AJS{2}\int_\Omega \left( |\vare(\bu)|^{q-2} \vare(\bu) - |\vare(\bu_h)|^{q-2} \vare(\bu_h) \right): \vare(\bu - \bU_h) \omega \diff x.}
\end{multline}
\EO{We thus invoke Young's inequality to obtain the bound
\begin{multline}
\label{eq:needforql3}
\AJS{\mu} \| \vare(\bu - \bu_h )\|_{\bL^2(\Omega)}^2 + \AJS{2}\int_\Omega \left( |\vare(\bu)|^{q-2} \vare(\bu) - |\vare(\bu_h)|^{q-2} \vare(\bu_h) \right): \vare(\bu - \bu_h) \omega \diff x \\
\leq
2 \mu \| \vare(\bu - \bU_h )\|_{\bL^2(\Omega)}^2 
+
\AJS{2}\int_\Omega \left( |\vare(\bu)|^{q-2} \vare(\bu) - |\vare(\bu_h)|^{q-2} \vare(\bu_h) \right): \vare(\bu - \bU_h) \omega \diff x.
\end{multline}

Now,} since $q \geq 2$, we \AJS{can use \cite[Lemma 4.4, \EO{Chapter I}]{MR1230384} (see also \cite[Theorem 5.3.3]{CiarletBook}) to deduce
\begin{equation}
\label{eq:NeededForUniqueness}
 \| \vare(\bu - \bu_h )\|_{\bL^q(\omega,\Omega)}^q \lesssim\int_\Omega \left( |\vare(\bu)|^{q-2} \vare(\bu) - |\vare(\bu_h)|^{q-2} \vare(\bu_h) \right) :\vare(\bu - \bu_h) \omega \diff x.
\end{equation}
In addition, \cite[Lemma 2.1, \EO{estimate (2.1a)}]{MR1192966} with $\delta = 0$ allows us to write
\begin{multline*}
  \int_\Omega \left( |\vare(\bu)|^{q-2} \vare(\bu) - |\vare(\bu_h)|^{q-2} \vare(\bu_h) \right): \vare(\bu - \bU_h) \omega \diff x \\
  \lesssim \int_\Omega |\vare(\bu-\bu_h)| \left( |\vare(\bu)| + |\vare(\bu_h)| \right)^{q-2} |\vare(\bu - \bU_h)| \omega \diff x.
\end{multline*}
These estimates can be combined to infer}
\begin{multline}
\label{eq:almostthere}
  \| \vare(\bu - \bu_h )\|_{\bL^2(\Omega)}^2 + \| \vare(\bu - \bu_h )\|_{\bL^q(\omega,\Omega)}^q \lesssim \| \vare(\bu - \bU_h) \|_{\bL^2(\Omega)}^2 \\ + 
  \int_\Omega |\vare(\bu-\bu_h)| \left( |\vare(\bu)| + |\vare(\bu_h)| \right)^{q-2} |\vare(\bu - \bU_h)| \omega \diff x =: \EO{\| \vare(\bu - \bU_h) \|_{\bL^2(\Omega)}^2 +\wp.}
\end{multline}
\EO{Let us examine $\wp$} in more detail. Using that $\bu$ and $\bu_h$ minimize $\calJ$ over $\bX_q(\omega)$ and $\bX_h$, respectively, we obtain that their $\bW^{1,q}(\omega,\Omega)$ norms are uniformly bounded with bounds that only depend on the problem data. Thus, H\"older's inequality implies that
\begin{align*}
  \wp &\leq
    \| |\vare(\bu)| + |\vare(\bu_h)| \|^{q-2}_{\bL^q(\omega,\Omega)
  } \| \vare(\bu - \bu_h) \|_{\bL^q(\omega,\Omega)} \| \vare(\bu - \bU_h) \|_{\bL^q(\omega,\Omega)} \\
  &\leq \gamma \| \vare(\bu - \bu_h) \|_{\bL^q(\omega,\Omega)}^q + C(\bef,\gamma, \AJS{q}) \| \vare(\bu - \bU_h) \|_{\bL^q(\omega,\Omega)}^{q/(q-1)}.
\end{align*}
Choosing $\gamma$ sufficiently small we can absorb the first term of this estimate for $\wp$ on the left hand side of \eqref{eq:almostthere}. \EO{This concludes the proof.}
\end{proof}

\EO{
\begin{remark}[approximation]
Notice that, since $q\geq 2$ and $\omega \in A_{q/2}$, we can use the stability of the Stokes projection that we proved in Corollary~\ref{cor:stabStokesProj} to invoke \cite[Corollary 4.2]{DOS:19} and conclude that Theorem~\ref{thm:errestSmago} implies best approximation properties
\end{remark}
}

As a corollary, we provide an error estimate for the Smagorinski model \eqref{eq:Smagorinski}.

\begin{corollary}[error estimate]
\label{col:errestSmago}
Assume that $\alpha \in (-1,\tfrac12)$. Let the pair $(\bu,\pe) \in \bX_3(\dist(\cdot,\partial\Omega)^\alpha) \times M_3(\dist(\cdot,\partial\Omega)^\alpha)$ solve \eqref{eq:nlStokes} with the stress tensor given by \eqref{eq:Smagorinski}, and $(\bu_h,\pe_h) \in \bX_h \times M_h$ be its finite element approximation. We thus have the following error estimate
\begin{multline*}
  \| \vare(\bu - \bu_h )\|_{\bL^2(\Omega)}^2 + \| \vare(\bu - \bu_h )\|_{\bL^3(\dist(\cdot,\partial\Omega)^\alpha,\Omega)}^3 \\ \AJS{\lesssim
  \| \vare(\bu - \bU_h) \|_{\bL^2(\Omega)}^2 + \| \vare(\bu - \bU_h) \|_{\bL^3(\dist(\cdot,\partial\Omega)^\alpha,\Omega)}^{3/2} },
\end{multline*}
where the hidden constant is independent of $\bu$, $\pe$, $h$, \EO{and $(\bU_h,P_h)$; the latter being the Stokes projection of $(\bu,\pe)$}.
\end{corollary}
\begin{proof}
According to \cite{MR3215609} and \cite[Lemma 2.3(vi)]{MR1601373}, in $d$ dimensions, if $\P$ denotes a $k$--dimensional compact Lipschitzian manifold, with $k \in \{0,1, \ldots d- 1\}$,
then $\dist(\cdot,\P)^\alpha \in A_t$ provided
\[
  \alpha \in ( -(d-k), (d-k)(t-1) ).
\]
In our setting $d=3$, $\P = \partial \Omega$, so that $k=2$, and $t = 3/2$. Thus, requiring $\alpha \in (-1,\tfrac12)$ guarantees that $\dist(\cdot,\partial\Omega)^\alpha \in A_{3/2}$. We can then apply Theorem~\ref{thm:errestSmago} to conclude.
\end{proof}

\begin{remark}[novelty]
Notice that, \EO{since} $\dist(\cdot,\partial\Omega)^\alpha \notin A_t(\Omega)$ for any $t$, \EO{our} new results, namely Theorem \ref{thm:velest} and Corollaries \ref{col:presest} and \ref{cor:stabStokesProj} are, once again, essential in deducing Corollary~\ref{col:errestSmago}.
\end{remark}

\begin{remark}[$q<2$]
\label{rem:qlesstwoSmago}
The error \EO{estimate} of Theorem~\ref{thm:errestSmago} can be extended to the case $q \in (1,2)$ provided $\omega \in A_{q'/2}$ \EO{upon repeating} the same arguments that lead to \eqref{eq:needforql3}. In this case, however,
\[
  \int_\Omega \left( |\vare(\bu)|^{q-2} \vare(\bu) - |\vare(\bu_h)|^{q-2} \vare(\bu_h) \right): \vare(\bu - \bu_h) \omega \diff x
\]
is dropped because it is nonnegative. \EO{Observe that the function $|x|^q$, with $q>1$, is convex and thus $(|x|^q- |y|^q)(x-y) \geq 0$ for every $x,y \in \mathbb{R}$.} 
 To deal with the last term on the right hand side of \eqref{eq:needforql3}, we apply \cite[Lemma 2.1]{MR1192966} with $\delta =1$ and obtain that
\begin{align*}
& \EO{\int_\Omega \left( |\vare(\bu)|^{q-2} \vare(\bu) - |\vare(\bu_h)|^{q-2} \vare(\bu_h) \right): \vare(\bu - \bU_h) \omega \diff x} 
\\
& \qquad
\lesssim \int_\Omega \left( |\vare(\bu)| + |\vare(\bu_h)| \right)^{q-1} |\vare(\bu-\bU_h)| \omega \diff x 
\\
& \qquad  
\leq \left( \| \vare(\bu) \|_{\bL^q(\omega,\Omega)} + \| \vare(\bu_h) \|_{\bL^q(\omega,\Omega)} \right)^{q-1} \| \vare(\bu- \bU_h)\|_{\bL^q(\omega,\Omega)},
\end{align*}
and argue, again, that the $\bW^{1,q}_0(\omega,\Omega)$ norms of $\bu$ and $\bu_h$ must be bounded by data. In conclusion, we have proved the estimate
\[
  \| \vare(\bu - \bu_h )\|_{\bL^2(\Omega)}^2 \lesssim \AJS{ \| \vare(\bu - \bU_h) \|_{\bL^2(\Omega)}^2 + \| \vare(\bu - \bU_h) \|_{\bL^q(\omega,\Omega)} }.
\]
\end{remark}

\subsection{Convection}
\label{sub:convectionSmago}

Let us, as a final extension, consider a generalization of \eqref{eq:nlStokes} with stress \eqref{eq:genSmagorinskivisco} that takes into account convection. \EO{This extends the analysis performed in \cite{MR3488119,MR3837915}. To be precise,} we consider the problem
\begin{equation}
\label{eq:Smagowithconv}
  \begin{dcases}
    -\DIV \polS(x,\varepsilon(\bu)) + (\bu\cdot\GRAD)\bu + \GRAD \pe = -\DIV \bef, & \text{ in } \Omega, \\
    \DIV \bu = 0, & \text{ in } \Omega, \\
    \bu = \boldsymbol0, & \text{ on } \partial \Omega.
  \end{dcases}
\end{equation}

The following is our main existence result.

\begin{theorem}[existence]
Let $\Omega \subset \R^3$ be a convex polyhedron, $q \in (2,\infty)$, $\omega \in A_q$, and $\bef \in \bL^2(\Omega) \AJS{+} \bL^{q'}(\omega',\Omega)$. Then, there is $(\bu,\pe) \in \bX_q(\omega) \times M_q(\omega)$ such that
\[
  \begin{dcases}
    \int_\Omega \polS(x,\vare(\bu)) :\vare(\bv) \diff x + \int_\Omega (\bu\cdot\GRAD)\bu \cdot \bv \diff x - \int_\Omega \pe \DIV \bv \diff x = \int_\Omega \bef: \GRAD \bv \diff x , \\
    \int_\Omega \DIV \bu r \diff x  = 0,
  \end{dcases}
\]
for all $(\AJS{\bv},q) \in \bX_q(\omega) \times M_q(\omega)$, where $\polS$ is given by \eqref{eq:genSmagorinskivisco}. In addition, if $\mu$ is sufficiently large, or $\bef$ sufficiently small, we obtain that $\bu$ is unique.
\end{theorem}
\begin{proof}
Recall the definition of $\bX_{q,\DIV}(\omega)$. Define the operator
\[
  \mathfrak{NL}:\bX_{q,\DIV}(\omega) \times \bX_{q,\DIV}(\omega) \to \bX_{q,\DIV}(\omega)'
\]
by
\[
  \langle \mathfrak{NL}(\bu,\bv),\bw \rangle = \int_\Omega (\bu\cdot\GRAD)\bv \cdot \bw \diff x.
\]
The fact that
\[
  \| \bw \|_{\bX_{q,\DIV}(\omega)} = \| \GRAD \bw \|_{\bL^2(\Omega)} + \| \GRAD \bw \|_{\bL^q(\omega,\Omega)}
\]
shows that this is a compact operator in $\bX_{q,\DIV}(\omega)$. A standard fixed point argument 
then yields existence of solutions.

Existence of the pressure follows the argument of Theorem~\ref{thm:Rappaz}. Notice also that testing with $\bv = \bu$ yields, \AJS{for any decomposition $\bef = \bef_1 + \bef_2$, \EO{with $\bef_1 \in \bL^2(\Omega)$ and $\bef_2 \in \bL^{q'}(\omega',\Omega)$,} the following estimate
\begin{multline}
\label{eq:stabforuniqueness}
  \mu \| \vare(\bu) \|_{\bL^2(\Omega)}^2 + \left( 1 + \frac1{q'} \right) \| \vare(\bu) \|_{\bL^q(\omega,\Omega)}^q  \\ \leq \frac{K(2,1)^2}{\EO{4}\mu}\|\bef_1\|_{\bL^2(\Omega)}^2 + \frac{K(q,\omega)^{q'}}{q'}\|\bef_2\|_{\bL^{q'}(\omega',\Omega)}^{q'},
\end{multline}
where we denoted by $K(q,\omega)$ the hidden constant in Korn's inequality \eqref{eq:WeightedKorn}}.

To obtain uniqueness with small data, we follow the classical uniqueness ideas for the Navier Stokes equations \cite[Theorem 2.1.3]{MR1846644}. Assuming that there are two solutions $\bu_1,\bu_2$, and denoting the difference $\bu = \bu_1 - \bu_2$ yields
\AJS{
\begin{align*}
  \int_\Omega \left( \polS(x,\vare(\bu_1)) - \polS(x,\vare(\bu_2)) \right):\vare(\bu)\diff x &= -\int_\Omega (\bu\cdot\GRAD)\bu_1\cdot\bu \diff x \\ &\lesssim \| \GRAD \bu \|_{\bL^2(\Omega)}^2 \| \GRAD \bu_1 \|_{\bL^2(\Omega)},
\end{align*}
where, per usual, we employed the skew symmetry of convection when the first argument is solenoidal
\[
  \int_\Omega (\bu_2\cdot \GRAD)\bu \cdot \bu \diff x = 0.
\]
Since $q >2$ a similar argument to the one that led to \eqref{eq:NeededForUniqueness} implies that
\[
  \int_\Omega \left( \polS(x,\vare(\bu_1)) - \polS(x,\vare(\bu_2)) \right):\vare(\bu)\diff x \gtrsim \| \vare(\bu) \|_{\bL^2(\Omega)}^2 + \| \vare(\bu) \|_{\bL^q(\omega,\Omega)}^q.
\]
This, combined with Korn's inequality \eqref{eq:WeightedKorn} for $q = 2$ and $\omega = 1$, and \eqref{eq:stabforuniqueness}, implies
\[
  \| \vare(\bu) \|_{\bL^2(\Omega)}^2 \lesssim \frac{K(2,1)^3}{\EO{\sqrt{\mu}}} \left[ \frac{K(2,1)^2}{\EO{4}\mu}\|\bef_1\|_{\bL^2(\Omega)}^2 + \frac{K(q,\omega)^{q'}}{q'}\|\bef_2\|_{\bL^{q'}(\omega',\Omega)}^{q'} \right]^{\EO{\frac{1}{2}}} \| \vare(\bu) \|_{\bL^2(\Omega)}^2.
\]
The assumption} that $\bef$ is sufficiently small or $\mu$ sufficiently large, allows us to absorb this term on the left hand side and conclude uniqueness.
\end{proof}

\section*{Acknowledgements}
The authors would like to thank Vivette Girault and Johnny Guzm\'an for interesting discussions and hints that led to the proof of Theorem~\ref{thm:stabclassic}. They would also like to thank Leo Rebholz for providing insight and useful references regarding the Smagorinski model.

\appendix

\section{Proof of Theorem~\ref{thm:stabclassic}}
\label{sec:GNS}

The purpose of this, supplementary, section is to detail what changes, if any, are necessary to translate the results of \cite[Theorem 4.1]{DOS:19} to the case that we are interested in here. We comment that the main difference here is that, in the first \EO{equation} of the definition of the Stokes projection, reference \cite{DOS:19} employs gradients (see \cite[formula $(1.1)$]{DOS:19}), while we employ symmetric gradients. While in the \EO{continuous} case this only amounted to a redefinition of the pressure, it rarely happens in practice that we have $\DIV \bX_h \subset M_h$ and so this change of variables cannot be performed.

Let us begin then with some notation. We define 
\[
  a(\bv,\bw) = \AJS{2}\mu \int_\Omega \vare(\bv):\vare(\bw) \diff x, \qquad b(\bv,q) = -\int_\Omega q \DIV \bv \diff x.
\]
We now realize that the heart of the matter in the proof of \cite[Theorem 4.1]{DOS:19} is the estimate provided in \cite[formula $(4.1)$]{DOS:19}. Thus, if we can prove 
\begin{equation}
\label{eq:4.1Duran}
  \| \vare(\bu_h) \|_{\bL^2(\omega,\Omega)} \lesssim \| \vare(\bu) \|_{\bL^2(\omega,\Omega)} + \| p \|_{L^2(\omega,\Omega)},
\end{equation}
the rest of the proof follows \emph{verbatim}. We \EO{thus} focus on the proof of \eqref{eq:4.1Duran}. This is derived in several steps.

\subsection{Approximate Green function}
\label{sub:ApproxGreen}

We begin the proof of \eqref{eq:4.1Duran} by defining suitable approximate Green functions. Let $z \in \Omega$ be such that $z \in \mathring{T}_z$, for some $T_z \in \T$, and $\tilde\delta_z$ be a regularized Dirac delta function that satisfies the properties:
\begin{enumerate}[$\bullet$]
  \item $\tilde \delta_z \in C_0^\infty(T_z)$;
  
  \item $\int_\Omega \tilde\delta_z \diff x = 1$;
  
  \item $\| \tilde \delta_z \|_{L^t(T_z)} \lesssim h^{-3/t'}$, for $t \in [1,\infty]$;
  
  \item $\int_\Omega \tilde\delta_z \bv_h \diff x = \bv_h(z)$ for all $\bv_h \in \bX_h$.
\end{enumerate}
The construction of such a regularized Dirac delta is presented in \cite[Section 1]{MR2121575}.

Let $z \in \Omega$ \AJS{be such that $z \in \mathring{T}_z$, for some $T_z \in \T$,} and $i,j \in \{1,2,3\}$. We define the approximate (derivative of the) Green function as the pair $(\frakG,\frakq) \in \bW^{1,2}_0(\Omega) \times \cL^2(\Omega)$ that solves
\begin{equation}
\label{eq:approxGreen}
  \begin{dcases}
    a(\frakG,\bv) + b(\bv,\frakq) = \int_\Omega \tilde \delta_z \vare(\bv)_{i,j} \diff x & \forall \bv \in \bW^{1,2}_0(\Omega), \\
    b(\frakG,q) = 0 & \forall q \in \cL^2(\Omega).
  \end{dcases}
\end{equation}
We also define the Stokes projection $(\frakG_h,\frakq_h) \in \bX_h \times M_h$ of $(\frakG,\frakq)$ via
\begin{equation}
\label{eq:approxGreenh}
  \begin{dcases}
    a(\frakG_h,\bv_h) + b(\bv_h,\frakq_h) = \int_\Omega \tilde \delta_z \vare(\bv_h)_{i,j} \diff x & \forall \bv_h \in \bX_h, \\
    b(\frakG_h,q_h) = 0 & \forall q_h \in M_h.
  \end{dcases}
\end{equation}

As one last ingredient, we introduce regularized distances. For $y \in \Omega$, we define
\[
  \sigma_y(x) = \left( |x-y|^2 + \kappa^2 h^2 \right)^{1/2},
\]
where \EO{$\kappa\geq1$} is independent of $h$ but must satisfy that $\kappa h \leq R$ where $R = \diam \Omega$. The properties of this weight are given in \cite[Section 1]{MR2121575} and \cite[Section 1.7]{MR342245}.

\subsection{Reduction to weighted estimates}
\label{sub:ReductionToWeight}

Having introduced the functions $(\frakG,\frakq)$ and their approximations $(\frakG_h,\frakq_h)$ we can proceed with the proof of \eqref{eq:4.1Duran}. Upon realizing that the only property of $a$ that is used in step 2 of the proof of \cite[Theorem 4.1]{DOS:19} is symmetry, we can follow the arguments without any change to arrive at
\begin{align*}
  \int_\Omega \omega |\vare(\bu_h)_{i,j}|^2 \diff z 
  &\lesssim 
  \int_\Omega \omega \left[\int_\Omega \vare(\bu):\vare(\bE) \diff x\right]^2 \diff \AJS{z} 
  + 
  \int_\Omega \omega \left[ \int_\Omega p \DIV \bE \diff x \right]^2 \diff z 
  \\ &
  + 
  \int_\Omega \omega \left[\fint_{T_z} |\vare(\bu)| \diff x \right]^2 \diff z,
\end{align*}
where we denoted $\bE = \frakG - \frakG_h$. In conclusion, \eqref{eq:4.1Duran} holds provided we can show that the estimate
\begin{equation}
\label{eq:ApproxGreenEstimate}
  \sup_{z \in \Omega} \| \sigma_z^{\mu/2} \vare( \frakG-\frakG_h) \|_{\bL^2(\Omega)} \lesssim h^{\lambda/2}
\end{equation}
holds for all $\nu \in (0,1/2)$, $\lambda \in (0,\nu/2)$ and $\mu = 3+\lambda$. \AJS{Notice that here we are following the notation of \cite{MR342245} so $\mu$ is not the viscosity}.

The rest of this Appendix is dedicated to indicate what changes, if any, are necessary to prove \eqref{eq:ApproxGreenEstimate}.

\subsection{Proof of \eqref{eq:ApproxGreenEstimate}}
\label{sub:GNSChanges}

Notice, first of all, that \eqref{eq:ApproxGreenEstimate} is the analogue of \cite[formula $(1.46)$]{MR342245}, so we follow this reference to indicate what changes are necessary. By changing gradients to symmetric gradients, where appropriate, we can reach the analogue of \cite[formula $(2.3)$]{MR342245}
\begin{align*}
  \int_\Omega \sigma_z^\mu |\vare(\bE)|^2 \diff x &= \int_\Omega \vare(\bE):\vare(\sigma_z^\mu(\frakG - P_h (\frakG) )) \diff x + \int_\Omega \vare(\bE):\vare(\bpsi - \bar P_h(\bpsi) ) \diff x \\
    &-\int_\Omega \GRAD \sigma_z^\mu \cdot (\vare(\bE) \bE) \diff x + \int_\Omega R \DIV \bar P_h(\bpsi) \diff x,
\end{align*}
where we set $R = \frakq - \frakq_h$, $\bpsi = \sigma_z^\mu( P_h (\frakG) - \frakG_h)$, and the interpolants $P_h$ and $\bar P_h$ are described in \cite[Section 1.8]{MR342245}.

We must now derive suitable weighted bounds on the pair $(\frakG,\frakq)$, as it is done in \cite[Section 2.4]{MR342245}. We just comment that:
\begin{enumerate}[$\bullet$]
  \item \cite[Proposition 2]{MR342245}, which is  in \cite[Propostion 3.1]{MR2121575}, follows without changes, so that we obtain
  \[
    \| \sigma_z^{\mu/2-1} \frakq \|_{L^2(\Omega)} \lesssim \| \sigma_z^{\mu/2-1} \vare(\frakG) \|_{\bL^2(\Omega)} + \kappa^{\mu/2-1} h^{\lambda/2-1}.
  \]
  
  \item \cite[Proposition 3]{MR342245}, which is  in \cite[Proposition 3.2]{MR2121575}, follows with little changes to obtain
  \[
    \| \sigma_z^{\mu/2-1} \vare(\frakG) \|_{\bL^2(\Omega)}^{\AJS{2}} \leq \| \sigma_z^{\mu/2-2} \frakG \|_{\bL^2(\Omega)} \left( c_1 \kappa^{\mu/2} h^{\lambda/2-1} +c_2\| \sigma_z^{\mu/2-1} \vare(\frakG) \|_{\bL^2(\Omega)} \right).
  \]
  
  \item \cite[Theorem 5]{MR342245}, follows without changes, so that we obtain the analogue of \cite[Corollary 1]{MR342245}:
  \[
    \| \sigma_z^{\mu/2-1} \vare(\frakG) \|_{\bL^2(\Omega)} + \| \sigma_z^{\mu/2-1} \frakq \|_{L^2(\Omega)}  \lesssim \kappa^{\mu/4}h^{\lambda/2-1}.
  \]
  
  \item To obtain the regularity estimates of \cite[Theorem 6]{MR342245}, we follow the arguments given in \cite[Theorem 3.6]{MR2121575} and realize that we must compute the effect of the Stokes operator on  $(\sigma_z^{\mu/2}\frakG,\sigma_z^{\mu/2}\frakq)$. After elementary computations, one realizes that
  \[
    -\DIV(\vare(\sigma_z^{\mu/2}\frakG)) + \GRAD(\sigma_z^{\mu/2}\frakq) = \sigma_z^{\mu/2}\left[-\DIV(\vare(\frakG)) + \GRAD \frakq\right] + \calF,
  \]
  where $\calF$ depends on $\frakG$, $\GRAD \frakG$, $\GRAD \sigma_z^{\mu/2}$, and $\frakq$. The important point is that $\calF \in \bL^2(\Omega)$. Using the regularity results for the Stokes operator of Proposition~\ref{prop:regular}, we conclude then that the right hand side of the expression above is an element of $\bL^2(\Omega)$. In addition, we have that
  \[
    \DIV(\sigma_z^{\mu/2}\frakG) = \GRAD \sigma_z^{\mu/2} \cdot \frakG \in W^{1,2}_0(\Omega) \cap \cL^2(\Omega).
  \]
  In conclusion, upon invoking Proposition~\ref{prop:regular} once again, we have that
  \[
    \| \sigma_z^{\mu/2} D^2\frakG \|_{\bL^2(\Omega)} + \| \sigma_z^{\mu/2} \GRAD \frakq \|_{\bL^2(\Omega)} \lesssim \kappa^{\mu/2}h^{\lambda/2-1}
  \]
  and \cite[Theorem 7]{MR342245} follows without changes.
\end{enumerate}

The discussion of \cite[Section 3]{MR342245} is about finite element spaces, and so it does not need any changes.

At this point we have set the stage to carry out the bootstrap procedure of \cite[Section 4]{MR342245}. To carry out the duality argument in the proof of Theorem 9, we must introduce the pair $(\bvphi,s) \in \bW^{1,2}_0(\Omega) \times \cL^2(\Omega)$ that solves
\[
  -\DIV(\vare(\bvphi)) + \GRAD s = \sigma_z^{\mu+\epsilon-2}(\frakG-\frakG_h), \quad \DIV \bvphi = 0, \ \text{ in } \Omega, \quad \bvphi = 0, \ \text{ on } \partial\Omega.
\]
The redefinition of the pressure indicated in Remark~\ref{rem:LapISDivVare} allows us to conclude, \AJS{owing to \cite{MR2228352}, that there is $\alpha \in (0,1)$ that depends on $\Omega$ and for which we have $(\bvphi,s) \in \bC^{1,\alpha}(\bar\Omega) \times C^{0,\alpha}(\bar\Omega)$} with an estimate similar to \cite[estimate (4.4)]{MR342245}. Upon replacing gradients by symmetric gradients in \cite[formulas (4.5), (4.6), and (4.7)]{MR342245} we arrive at the analogue of \cite[estimate (4.8)]{MR342245}:
\begin{multline*}
  \left\| \sigma_z^{\frac12(\mu+\epsilon)-1}(\frakG - \frakG_h) \right\|_{\bL^2(\Omega)}^2 \leq 
    \| \sigma_z^{-\mu/2} \vare(\bvphi - \bar P_h(\bvphi)) \|_{\bL^2(\Omega)} \| \sigma_z^{\mu/2}\vare(\frakG - \frakG_h ) \|_{\bL^2(\Omega)}  \\
    + \| \sigma_z^{-\mu/2} \DIV(\bvphi - \bar P_h(\bvphi)) \|_{L^2(\Omega)} \| \sigma_z^{\mu/2} (\frakq - r_h(\frakq)) \|_{L^2(\Omega)} \\
    \AJS{-} \int_\Omega (s - \bar r_h(s) ) \DIV(\frakG-\frakG_h) \diff x = \mathrm{I} + \mathrm{II} + \mathrm{III},
\end{multline*}
where \AJS{$\epsilon \in (0,1)$ is as in \cite[(4.8)]{MR342245},} $r_h$  and $\bar{r}_h$ are the interpolants described in \cite[Section 1.8]{MR342245}.
\AJS{Let us now show the slight departures from the argument in \cite{MR342245} with some detail}. As in \cite{MR342245}, the regularity of $(\bvphi,s)$ implies that
\begin{multline*}
  \| \sigma_z^{-\mu/2} \vare(\bvphi - \bar P_h(\bvphi)) \|_{\bL^2(\Omega)} + \| \sigma_z^{-\mu/2} \DIV(\bvphi - \bar P_h(\bvphi)) \|_{L^2(\Omega)} \\
    + \| \sigma_z^{-\mu/2}(s - \bar r_h(s)) \|_{L^2(\Omega)} \lesssim h^\alpha (\kappa h)^{-\lambda/2} \left\| \sigma_z^{\mu +\epsilon -2}(\frakG - \frakG_h) \right\|_{\bL^r(\Omega)},
\end{multline*}
where $\alpha = 1-3/r$. The estimate on terms $\mathrm{I}$ and $\mathrm{II}$ then proceeds as in \cite{MR342245}. The estimate of $\mathrm{III}$ is \AJS{obtained by noticing that, for any matrix $M$, $|\tr M| \lesssim | M |$}. In addition, for any vector field $\bv$, we have $\tr \vare(\bv) = \DIV \bv$. These observations allow us to write
\begin{align*}
  \mathrm{III}& \leq \| \sigma_z^{-\mu/2}(s - \bar r_h(s)) \|_{L^2(\Omega)} \|\sigma_z^{\mu/2} \DIV(\frakG - \frakG_h)\|_{L^2(\Omega)} \\
  &\lesssim \| \sigma_z^{-\mu/2}(s - \bar r_h(s)) \|_{L^2(\Omega)} \|\sigma_z^{\mu/2} \vare(\frakG - \frakG_h)\|_{\bL^2(\Omega)},
\end{align*}
and the estimate on $\mathrm{III}$ now proceeds as in \cite{MR342245}.

We have thus arrived at \cite[(4.10)]{MR342245}, where one can \AJS{Korn's inequality \eqref{eq:WeightedKorn} with $q=2$ and $\omega \equiv 1$}, to replace the gradient by a symmetric gradient. We can keep replacing gradients by $\vare$ to arrive at the conclusion of \cite[Theorem 9]{MR342245}.

Once \cite[Theorem 9]{MR342245} holds, \cite[Corollaries 2 and 3]{MR342245} are a simple exercise and thus we obtain
\[
  \| \sigma_z^{\mu/2-1} (\frakG- \frakG_h) \|_{\bL^2(\Omega)}^2 \lesssim \frac1{\kappa^\alpha} \| \sigma_z^{\mu/2}\vare(\frakG - \frakG_h) \|_{\bL^2(\Omega)}^2 + \kappa^{\mu-\alpha} h^\lambda,
\]
and
\[
  \left\| \sigma_z^{\frac12(\mu+\epsilon)-1}(\frakG - \frakG_h) \right\|_{\bL^2(\Omega)}^2 \AJS{\leq \frac{(\kappa h)^\epsilon}{\kappa^\alpha} \left( C_\kappa \| \sigma_z^{\mu/2} \vare(\frakG- \frakG_h) \|_{\bL^2(\Omega)}^2 + \kappa^\mu h^\lambda \right)},
\]
\AJS{where $C_\kappa := 1 + \tfrac1{\kappa^\alpha}$}.

The estimates on the pressure term \cite[Section 5]{MR342245} hold with little or no modification. In \cite[Lemmas 7 and 8]{MR342245} one only needs to replace gradients with symmetric gradients. The same is true for \cite[Theorem 10]{MR342245}, which is \cite[Theorem 4.2]{MR2121575}. Then, \cite[Proposition 10 and Theorem 11]{MR342245} need no changes.

This, finally, brings us to \cite[Section 6]{MR342245}, where \cite[Theorem 12]{MR342245} proceeds without changes, proves \eqref{eq:ApproxGreenEstimate}, and concludes our argument.

As a final remark, we comment that in this new setting \cite[Theorem 13 and Corollary 4]{MR342245} also follow without changes. The pressure estimates of \cite[Section 6.2]{MR342245} only require to change \cite[(6.6)]{MR342245} accordingly.


\bibliographystyle{amsplain}
\bibliography{biblio}

\end{document}